\documentclass[12pt]{article}

\usepackage{amsmath}
\usepackage{amsfonts}
\usepackage{amssymb}
\usepackage{graphicx}
\usepackage{gauss}
\usepackage{amsthm}
\usepackage{bm}
\usepackage{lscape}
\usepackage{enumerate}
\usepackage{hyperref}
\usepackage{caption,subcaption}
\usepackage{color}
\usepackage{float}

\usepackage{tikz}
\usetikzlibrary{automata,positioning}
\usepackage[margin=1.5cm]{caption}
\usepackage{subcaption}
	\usetikzlibrary{positioning}
	\usetikzlibrary{arrows}
        \tikzset{%
        fwdrxn/.style={very thick, arrows={-Stealth[length=5pt,width=5pt]}},
        revrxn/.style={very thick, arrows={-Stealth[length=5pt,width=5pt,left]}},
        newt/.style={turq, opacity=0.15}
        }
        \tikzset{near start abs-right/.style={xshift=1cm}}
        \tikzset{near start abs-left/.style={xshift=-3.5cm}}
        \tikzset{near start abs-up/.style={yshift=1.5cm}}
        \tikzset{near start abs-down/.style={yshift=-1cm}}

\graphicspath{{./figures/}} 

\newtheorem{theorem}{Theorem}[section]
\newtheorem{lemma}[theorem]{Lemma}

\newtheorem{proposition}[theorem]{Proposition}
\newtheorem{remark}[theorem]{Remark}
\newtheorem{remarks}[theorem]{Remarks}
\newtheorem{example}[theorem]{Example}

\theoremstyle{definition}
\newtheorem{definition}[theorem]{Definition}

\newcommand{\ignore}[1]{}

\newcommand{\END}{\hfill\mbox{\raggedright$\Diamond$}}
\newcommand{\etal}{\textrm{et al.}}

\newcommand*{\doi}[1]{\href{http://dx.doi.org/#1}{doi: #1}}

\newcommand{\Matrix}[1]{\ensuremath{\left[\begin{array}{rrrrrrrrrrrrrrrrrr} #1 \end{array}\right]}}
\newcommand{\Matrixc}[1]{\ensuremath{\left[\begin{array}{cccccccccccc|ccc} #1 \end{array}\right]}}

\numberwithin{equation}{section}

\renewcommand{\epsilon}{\varepsilon}

\newcommand{\R}{{\mathbf R}}

\renewcommand{\AA}{{\mathcal A}}

\newcommand{\CC}{{\mathcal C}}

\newcommand{\EE}{{\mathcal E}}

\newcommand{\GG}{{\mathcal G}}

\newcommand{\II}{{\mathcal I}}
\newcommand{\KK}{{\mathcal K}}

\newcommand{\NN}{{\mathcal N}}

\newcommand{\PP}{{\mathcal P}}

\newcommand{\LL}{{\mathcal L}}

\newcommand{\sS}{{\mathcal S}}

\newcommand{\VV}{{\mathcal V}} %

\newcommand{\IRB}{I_{RB}}
\newcommand{\Ii}{i_I}
\newcommand{\IE}{I_E}
\newcommand{\RF}{r_f}
\newcommand{\VE}{v_E}
\newcommand{\RI}{r_I}
\newcommand{\MR}{m_r}
\newcommand{\C}{c}
\newcommand{\VRB}{v_{RB}}
\newcommand{\VI}{v_I}
\newcommand{\MH}{{m_h}}
\newcommand{\h}{h}

\newcommand{\til}[1]{\widetilde{#1}}

\newcommand{\pathto}{\rightsquigarrow}

\makeatletter
\def\and{%
  \end{tabular}%
  \hskip 2.5em \@plus.17fil\relax
  \begin{tabular}[t]{c}}
\makeatother

\begin{document}

\title{Homeostasis Patterns}

\author{
William Duncan\renewcommand{\thefootnote}{\arabic{footnote}}\footnotemark[1] \and 
Fernando Antoneli\renewcommand{\thefootnote}{\arabic{footnote}}\footnotemark[2] 
\textsuperscript{,}\renewcommand{\thefootnote}{\fnsymbol{footnote}}\footnotemark[1] 
\and 
Janet Best\renewcommand{\thefootnote}{\arabic{footnote}}\footnotemark[3] \and
Martin Golubitsky\renewcommand{\thefootnote}{\arabic{footnote}}\footnotemark[3] \and 
Jiaxin Jin\renewcommand{\thefootnote}{\arabic{footnote}}\footnotemark[3]
\textsuperscript{,}\renewcommand{\thefootnote}{\fnsymbol{footnote}}\footnotemark[1] \and 
H. Frederik Nijhout\renewcommand{\thefootnote}{\arabic{footnote}}\footnotemark[4] \and 
Mike Reed\renewcommand{\thefootnote}{\arabic{footnote}}\footnotemark[5] \and 
Ian Stewart\renewcommand{\thefootnote}{\arabic{footnote}}\footnotemark[6]
}

\date{\today}

\maketitle

\renewcommand{\thefootnote}{\arabic{footnote}}
\footnotetext[1]{Simulations Plus, Lancaster, CA, USA}
\footnotetext[2]{Centro de Bioinform\'atica M\'edica, Universidade Federal de S\~ao Paulo, S\~ao Paulo, SP, Brazil}
\footnotetext[3]{Department of Mathematics, The Ohio State University, Columbus, OH, USA}
\footnotetext[4]{Department of Biology, Duke University, Durham, NC, USA}
\footnotetext[5]{Department of Mathematics, Duke University, Durham, NC, USA}
\footnotetext[6]{Mathematics Institute, University of Warwick, Coventry, UK}

\renewcommand{\thefootnote}{\fnsymbol{footnote}}
\footnotetext[1]{Correspondence:
\href{mailto:jin.1307@osu.edu}{jin.1307@osu.edu},
\href{mailto:fernando.antoneli@unifesp.br}{fernando.antoneli@unifesp.br}
}

\begin{abstract}
Homeostasis is a regulatory mechanism that keeps a specific variable close to a set value as other variables fluctuate.  The notion of homeostasis can be rigorously formulated when the model of interest is represented as an input-output network, with distinguished {\em input} and {\em output} nodes, and the dynamics of the network determines the corresponding {\em input-output function} of the system.  In this context, homeostasis can be defined as an `infinitesimal' notion, namely, the derivative of the input-output function is zero at an isolated point.  Combining this approach with graph-theoretic ideas from combinatorial matrix theory provides a systematic framework for calculating homeostasis points in models and classifying the different homeostasis types in input-output networks.  In this paper we extend this theory by introducing the notion of a {\em homeostasis pattern}, defined as a set of nodes, in addition to the output node, that are simultaneously infinitesimally homeostatic.  We prove that each homeostasis type leads to a distinct homeostasis pattern.  Moreover, we describe all homeostasis  patterns supported by a given input-output network in terms of a combinatorial structure associated to the input-output network.  We call this structure the {\em homeostasis pattern network}.

\medskip

\end{abstract}

\newpage

\tableofcontents

\section{Introduction}
\label{intro}

In biology, `homeostasis' originally referred to the ability of an organism to maintain a specific internal state despite varying external factors.  A typical example is the regulation of body temperature in a mammal despite variations in the temperature of its environment.  This concept goes back to 1849 when the French physiologist Claude Bernard observed this kind of regulation in the `milieu int\'erieur' (internal environment) of human organs such as the liver and pancreas; see the modern translation \cite{B49}.  The American physiologist Walter Cannon \cite{C26} developed this idea, coining the word `homeostasis' in 1926. The same basic concept has now spread to many areas of science.

In the literature, `homeostasis' is often modeled using differential equations, and is interpreted in two mathematically distinct ways.  One boils down to `stable equilibrium'. Here changes in the environment are considered to be perturbations of {\em initial conditions}.  A stronger (and, in our view, more appropriate) usage works with a parametrized family of differential equations, with a corresponding family of stable equilibria.  Now `homeostasis' means that this equilibrium changes by a relatively small amount when the {\em parameter} varies by a much larger amount.

In this paper we adopt the second, stronger, interpretation.  We also focus on the mathematical aspects of this concept. We say that {\em homeostasis} occurs in a system of differential equations when the output from the system $x_o$ is approximately constant on variation of an input parameter $\II$.  Golubitsky and Stewart~\cite{GS17} observe that homeostasis on some neighborhood of a specific value $\II_0$ follows from {\em infinitesimal homeostasis}, where $x_o'(\II_0) = 0$ and $'$ indicates differentiation with respect to $\II$.  This observation is essentially the well-known idea that the value of a function changes most slowly near a stationary (or critical) point.

\begin{remarks}\rm $ $
\begin{enumerate}[(a)]
\item Despite the name, `infinitesimal' homeostasis often implies that the system is homeostatic over a relatively large interval of the parameter \cite[Section 5.4]{GS23}.  The key quantity is the value of the {\em second} derivative $x_o''(\II_0)$ at the point $\II_0$.

\item Infinitesimal homeostasis is a {\em sufficient} condition for homeostasis over some interval of parameters, but it is not a necessary condition.  A function can vary slowly without having a stationary point.

\item One advantage of considering infinitesimal homeostasis is that it has a precise mathematical formulation, which makes it suitable for analysis using methods from singularity theory.  `Not varying by much' is a vaguer notion.

\item In applications, the quantity that experiences homeostasis can be a function of several internal variables, such as a sum of concentrations, or the frequency of an oscillation.  We do not consider such examples here, but similar `infinitesimal' methods might be developed for such cases.

\item Control-theoretic models of homeostasis often generate {\em perfect homeostasis}
(or {\em robust perfect adaptation} \cite{MA09,TM16,F16,K21,ST23}),

in which the equilibrium is exactly constant over the parameter range.  We do not adopt such a strong definition, in part because biological systems lack such precision.
However, everything we prove here can be applied to {\em perfect homeostasis}, since it is a particular case of infinitesimal homeostasis \cite{MA22}.
\END
\end{enumerate}
\end{remarks}

Wang \etal~\cite{WHAG21} consider infinitesimal homeostasis for a general class of {\em input-output networks} $\GG$.  Such a network has two distinguished nodes: the {\em input node} $\iota$, the only node that is affected by the input parameter $\II$, and the {\em output node} $o$.  To each fixed input-output network $\GG$ there is associated a space of admissible families of ODEs (or vector fields).  An admissible family of ODEs with a linearly stable family of equilibrium points defines an {\em input-output function} $x_o(\II)$ for $\GG$ (and the family of equilibria).  In \cite{WHAG21} it is shown that the derivative of $x_o(\II)$ with respect to $\II$ is given in terms of the determinant of the {\em homeostasis matrix} $H$ (see \eqref{eq:J_and_H}).  This `determinant formula' implies that input-output networks support infinitesimal homeostasis through a small number of distinct `mechanisms', called {\em homeostasis types}.

In this paper we consider the notion of a homeostasis pattern on a given input-output network $\GG$.  A {\em homeostasis pattern} is the set of nodes $j$ in $\GG$ (including the output node $o$) such that the node coordinate $x_j$, as a function of $\II$, satisfies $x'_j(\II_0) = 0$.  In other words, a homeostasis pattern is a set of nodes $\sS$ of $\GG$, that includes the output node $o$, and all nodes in $\sS$ are simultaneously (infinitesimally) homeostatic at a given parameter value $\II_0$.  The main result is that all the homeostasis patterns supported by a given input-output network $\GG$ can be completely classified in terms of the homeostasis types of $\GG$.

Consider for example the $6$ node input-output network $\GG$ shown in Figure \ref{F:example8}.

\begin{figure}[!htb]
\centerline{8
\includegraphics[width = .4\textwidth]{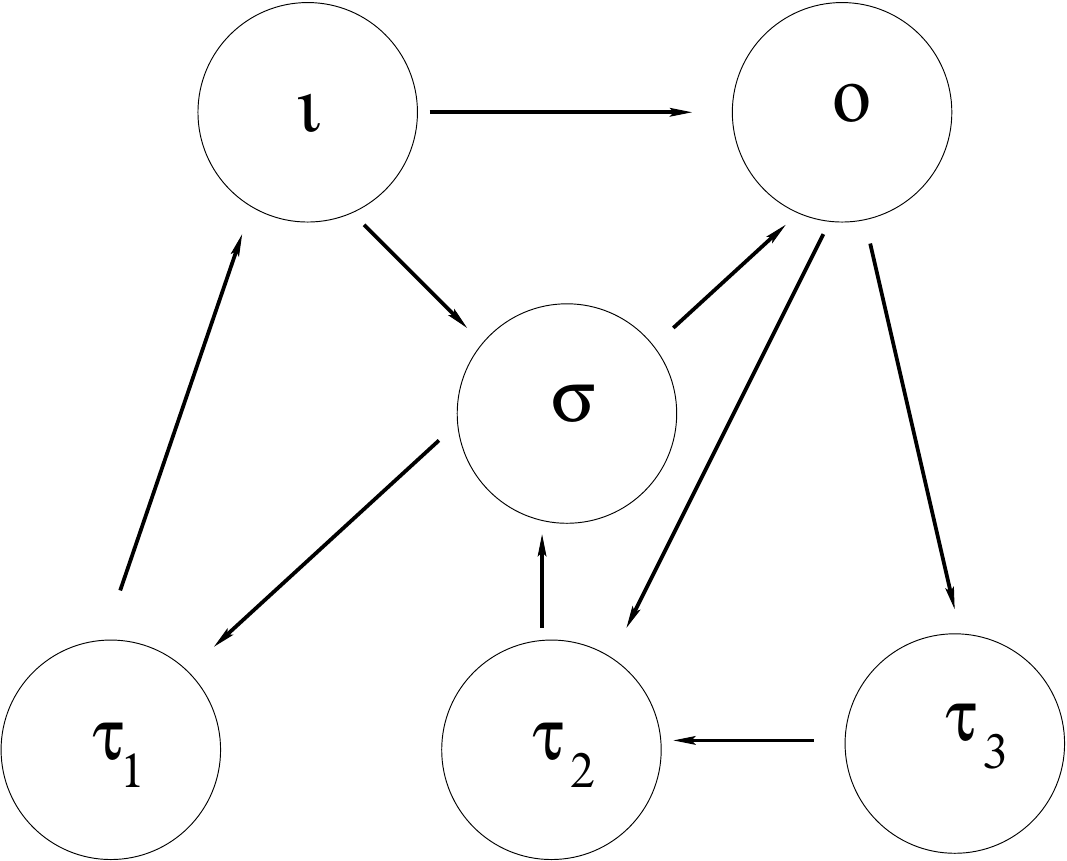}}
\caption{A $6$-node input-output network. \label{F:example8}}
\end{figure}

Although there are exactly $31$ subsets of nodes of $\GG$ including the output node $o$, only $4$ subsets define homeostasis patterns: $\{o\}$, $\{o, \tau_3\}$, $\{o, \tau_2, \tau_3\}$ and
$\{o, \tau_2, \tau_3, \sigma, \iota\}$.  These homeostasis patterns can be graphically represented by coloring the nodes of $\GG$ that are homeostatic (see Figure \ref{F:admissible}).

In this paper we lay out a general theory to classify all homeostasis patterns in a given input-output network. This method is purely combinatorial, based on the topology of the network, and does not rely on calculations involving the admissible ODEs. 
However, before going into the details of this theory, we use such calculations to give some indication of why the input-output network in Figure \ref{F:example8} has exactly the $4$ homeostasis patterns exhibited in Figure \ref{F:admissible}.  We do this using the results of~\cite{WHAG21}; see also subsection \ref{intro:io_networks}.

\begin{figure}[!htb]
\begin{subfigure}[c]{0.5\textwidth}
\centering
\includegraphics[width =0.6\textwidth]{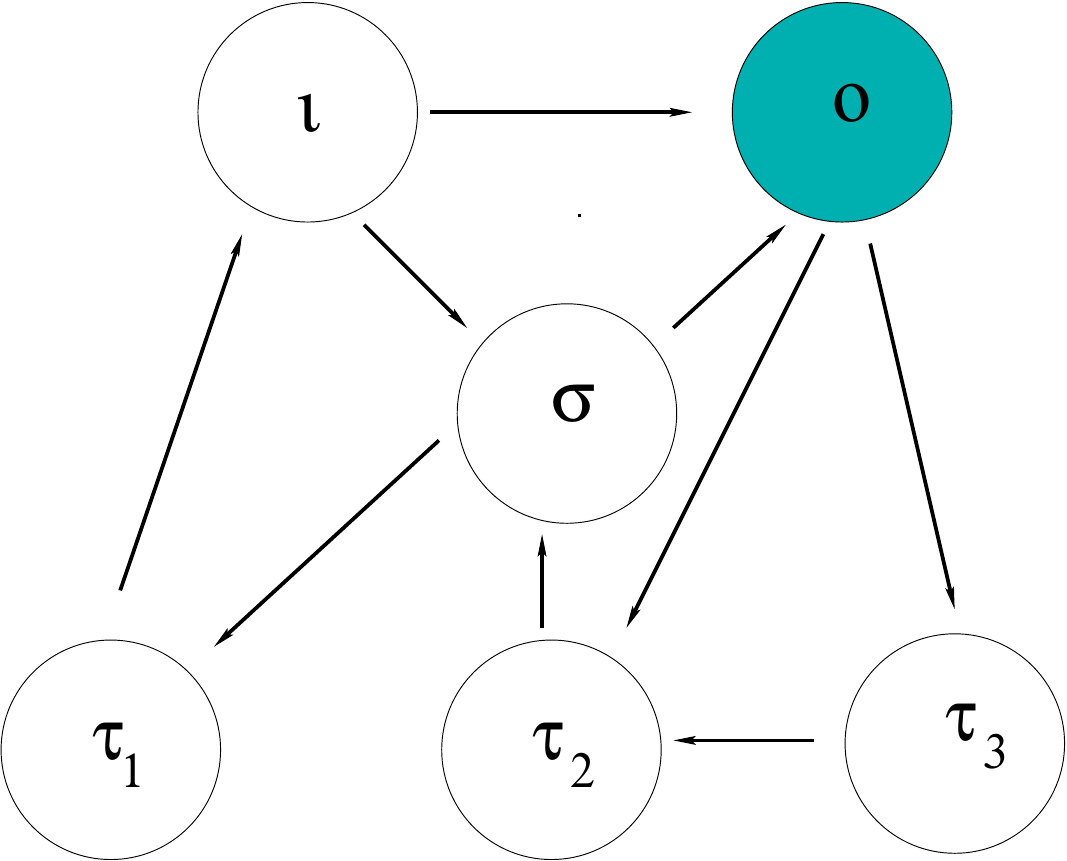}
\caption{$f_{\tau_3,\tau_3}(\II_0) = 0$}
\medskip\medskip
\end{subfigure}
\begin{subfigure}[c]{0.5\textwidth}
\centering
\includegraphics[width =0.6\textwidth]{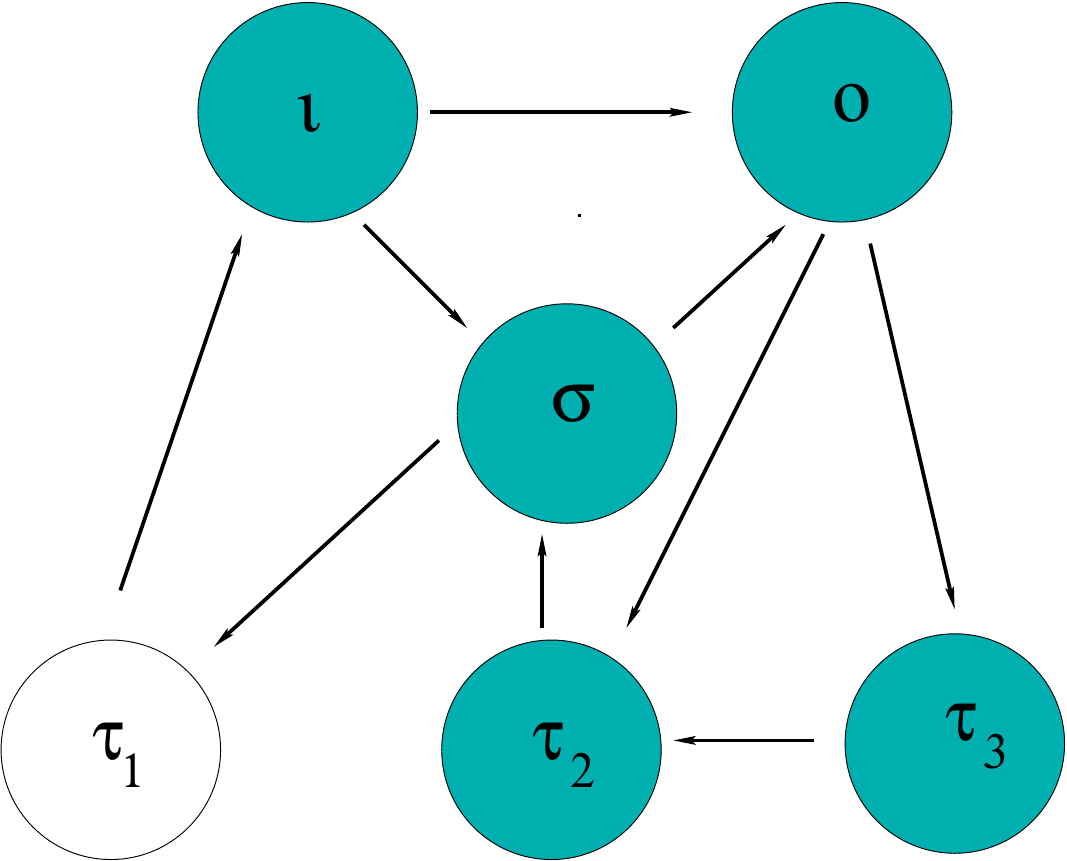}
\caption{$f_{\tau_1,\tau_1}(\II_0) = 0$}
\medskip\medskip
\end{subfigure}
\begin{subfigure}[c]{0.5\textwidth}
\centering
\includegraphics[width =0.6\textwidth]{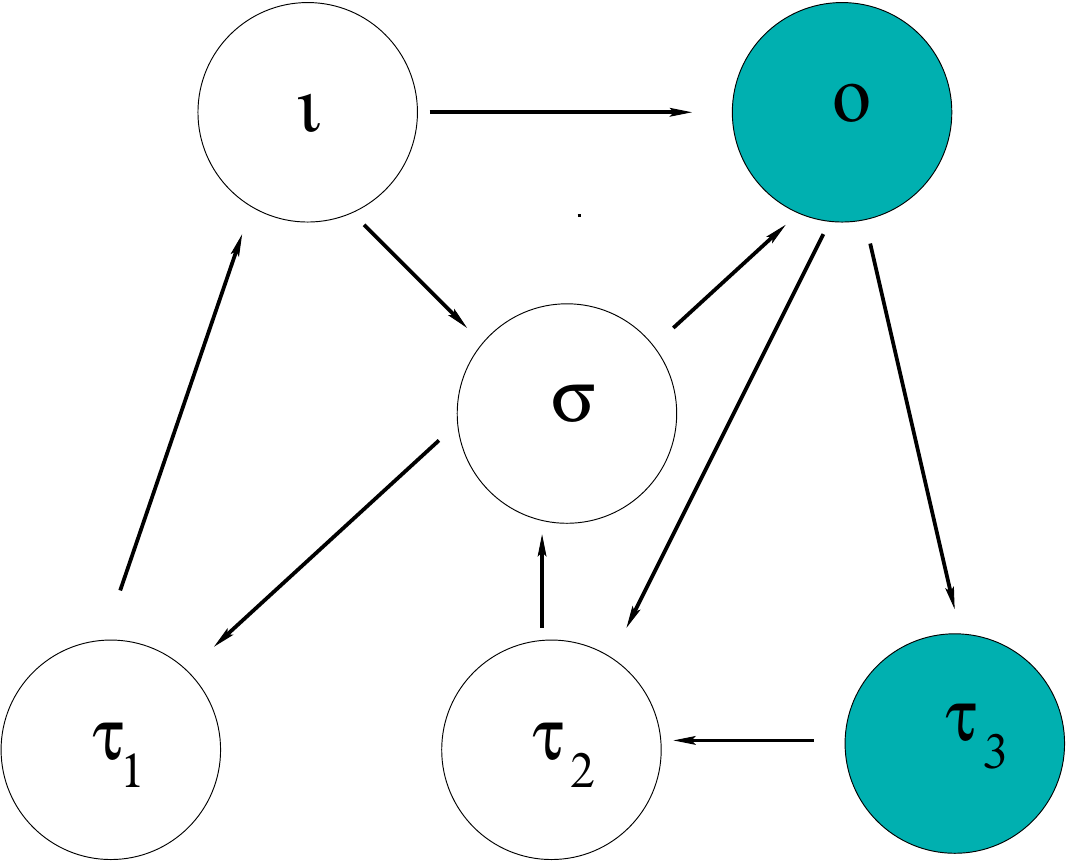}
\caption{$f_{\tau_2,\tau_2}(\II_0) = 0$}
\end{subfigure}
\begin{subfigure}[c]{0.5\textwidth}
\centering
\includegraphics[width =0.6\textwidth]{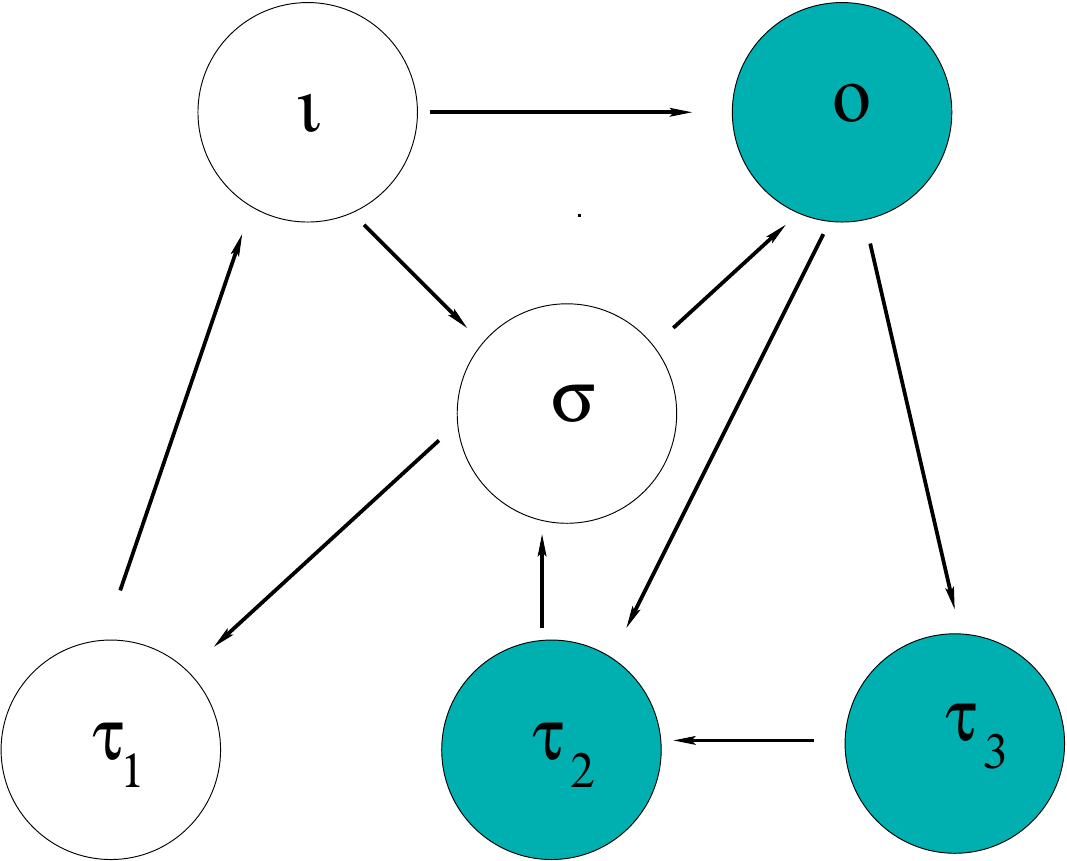}
\caption{$(f_{\sigma,\iota} f_{o,\sigma} - f_{\sigma,\sigma} f_{o,\iota})(\II_0) = 0$}
\end{subfigure}
\caption{The four infinitesimal homeostasis patterns of \eqref{e:admissible}.  Cyan nodes are homeostatic. \label{F:admissible}}
\end{figure}

The admissible system of parametrized equations for the network in Figure~\ref{F:example8}, in coordinates  $X = (\iota, \sigma, \tau_1,\tau_2, \tau_3, o)$, is:
\begin{equation} \label{e:admissible}
\begin{array}{ccl}
\dot{\iota} & = & f_\iota(\iota, \tau_1,\II) \\
\dot{\sigma} & = & f_\sigma(\iota, \sigma, \tau_2) \\
\dot{\tau}_1 & = & f_{\tau_1}(\sigma,\tau_1) \\
\dot{\tau}_2 & = & f_{\tau_2}(\tau_2,\tau_3,o) \\
\dot{\tau}_3 & = & f_{\tau_3}(\tau_3,o) \\
\dot{o} & = & f_o(\iota, \sigma,o)
\end{array}
\end{equation}

\pagebreak

The homeostasis matrix $H$ is obtained from the Jacobian matrix $J$ of \eqref{e:admissible} by removing the first row and the last column (see subsection \ref{intro:io_networks}, eq. \eqref{eq:J_and_H}).  This leads to:
\[
H = 
\Matrixc{f_{\sigma,\iota} & f_{\sigma,\sigma} & 0 & f_{\sigma,\tau_2} & 0\\  
 0 & f_{\tau_1, \sigma} &  f_{\tau_1, \tau_1} & 0 & 0 & \\ 
 0 & 0 & 0 & f_{\tau_2,\tau_2} & f_{\tau_2,\tau_3} \\  
0 & 0 & 0 & 0& f_{\tau_3,\tau_3} \\
f_{o,\iota} & f_{o,\sigma} & 0 & 0 & 0 
}
\]
Using row and column expansion it is straightforward to calculate
\begin{equation}  \label{H:example8}
\det(H) = f_{\tau_3,\tau_3} \, f_{\tau_1,\tau_1} \, f_{\tau_2,\tau_2} \,
(f_{\sigma,\iota} f_{o,\sigma} - f_{\sigma,\sigma} f_{o,\iota})  
\end{equation}
The `determinant formula' from~\cite{WHAG21} says that the input-output function $x_o(\II)$ undergoes infinitesimal homeostasis at $\II_0$ if and only if $\det(H) = 0$, evaluated at $(X(\II_0),\II_0)$.  Here $X(\II_0)$ is the equilibrium used to construct the input-output function (see subsection \ref{intro:io_networks}, Lemma \ref{lem:irreducible}).  The expression $\det(H)$ is a multivariate polynomial in the partial derivatives $f_{j,\ell}$ of the components of the admissible vector field.  As a polynomial, $\det(H)$ is reducible with $4$ irreducible factors, so $\det(H)=0$ if and only if one of its irreducible factors vanishes.

According to~\cite{WHAG21} these irreducible factors determine the homeostasis types
 (see subsection \ref{intro:io_networks}, Theorem~\ref{thm:irreducible}). Hence, \eqref{e:admissible} has $4$ homeostasis types.  One of the main results of this paper says that each homeostasis type determines a unique homeostasis pattern (see Theorem  \ref{thm:bijective_corresp_intro}).  Moreover, our theory gives a purely combinatorial procedure to find the set of nodes that belong to each homeostasis pattern.  When applied to Figure \ref{F:example8} it yields the four patterns in Figure \ref{F:admissible}.

In a simple example such as Figure \ref{F:example8} it is also possible to find the sets of nodes in each homeostasis pattern by a `bare hands' calculation based on the admissible ODEs (see subsection~\ref{intro:example}).  Here this calculation serves as a check on the results.  Direct calculations with the ODEs can, of course, be used instead of the combinatorial approach in sufficiently simple cases.

The rest of this introduction divides into subsections.  In subsection \ref{intro:io_networks} we define the admissible differential equations that are associated to input-output networks and infinitesimal homeostasis.  In subsection \ref{intro:type_pattern} we define homeostasis types and homeostasis patterns.  In subsection \ref{intro:example} we determine the homeostasis patterns of the network shown in Figure \ref{F:example8}.

\subsection{Input-Output Networks and Infinitesimal Homeostasis Types}
\label{intro:io_networks}

We begin by introducing the basic objects: \emph{input-output networks}, \emph{network admissible differential equations}, and \emph{infinitesimal homeostasis types}. Our exposition follows \cite{WHAG21}.  

\begin{definition}[{\cite[Section~1.2]{WHAG21}}] \label{defn:io_network}
An \emph{input-output network} is a directed graph $\GG$ with nodes $\kappa\in\CC$, arrows in $\EE$ 
connecting nodes in $\CC$, a distinguished input node $\iota$, and a distinguished output node $o$. 
The network $\GG$ is a \emph{core network} if every node in $\GG$ is downstream from $\iota$ and upstream from $o$. \END
\end{definition}

An {\em admissible system} of differential equations associated with $\GG$ has the form
\begin{align}\label{eq:system}
    \begin{split}
        \dot{x}_\iota &= f_\iota(x_\iota,x_\kappa,x_o,\II) \\
        \dot{x}_\kappa &= f_\kappa(x_\iota,x_\kappa,x_o) \\
        \dot{x}_o &= f_o(x_\iota,x_\kappa,x_o)
    \end{split}
\end{align}
where $\II\in\R$ is the input parameter, $X = (x_\iota,x_\kappa,x_o) \in \R\times \R^n \times \R$ is the vector of state variables associated to nodes in $\CC$, and $f(X,\II) = (f_\iota(X,\II),f_\kappa(X),f_o(X))$ is a smooth family of 
mappings on the state space $\R\times \R^n \times \R$. Note that $\II$ appears only in the equation of system \eqref{eq:system} corresponding to the input node.

We can write \eqref{eq:system} as 
\[
    \dot{X} = f(X,\II).
\]
We denote the partial derivative of the function associated to node $j$ with respect to the state variable associated to node $\ell$ by 
\[
    f_{j,\ell} = \frac{\partial}{\partial x_\ell} f_j
\]
We assume $f_{j,\ell} \equiv 0$ precisely when no arrow connects node $\ell$ to node $j$. That is, $f_j$ is independent of $x_\ell$ when there is no arrow $\ell\to j$.  This is a modeling assumption made in $\GG$.

Suppose $\dot{X} = f(X,\II_0)$ has a hyperbolic equilibrium at $X_0$. Then the implicit function theorem implies that there is a unique family of equilibria $X(\II) = (x_\iota(\II),x_\kappa(\II),x_o(\II))$ such that $X(\II_0) = X_0$ and $f(X(\II),\II) = 0$ for all $\II$ near $\II_0$. 

\begin{definition}
The mapping $\II \mapsto x_o(\II)$ is an \emph{input-output function}, which is defined on a neighborhood of $\II_0$. \emph{Infinitesimal homeostasis} occurs at $\II_0$ if $x_o'(\II_0) = 0$ where $'$ indicates differentiation with 
respect to $\II$. 
\begin{enumerate}[(a)]
\item If $x_o'(\II_0) = 0$ and $x_o''(\II_0)\neq 0$, then $o$ has a \emph{simple homeostasis point} at $(X_0,\II_0)$. 
\item If $x_o'(\II_0) = x_o''(\II_0)  = 0$ and $x_o'''(\II_0)\neq 0$, then $o$ has a \emph{chair point} at $(X_0,\II_0)$. \END
\end{enumerate}
\end{definition}
Nijhout \etal~\cite{NBR14} associated homeostasis with chairs, defined as a curve that is monotone except for a flat section.
The infinitesimal notion of a `chair point' was introduced in \cite{GS17}.

\begin{remark} \normalfont
Each node in an input-output network $\GG$ corresponds to a one-dimensional state variable of the admissible system.
In particular, the output node corresponds to a scalar quantity and the input parameter is a scalar quantity. 
This class of systems considered in this paper is also known as {\em single input, single output} (SISO) systems.
It is possible to consider input-output networks with multiple input nodes, but a single input parameter~\cite{MA22} and single output, and multiple inputs and single output~\cite{MA23}. \END
\end{remark}

\subsection{Homeostasis Type and Homeostasis Pattern}
\label{intro:type_pattern}

Wang~\etal~\cite{WHAG21} show that infinitesimal homeostasis occurs when the determinant of the homeostasis matrix $H$ is $0$, where the $(n+1)\times (n+1)$ matrix $H$ is obtained from the $(n+2)\times(n+2)$ Jacobian matrix $J$ of \eqref{eq:system} by deleting its first row and last column. Indeed
\begin{align}\label{eq:J_and_H}
J = \Matrix{f_{\iota,\iota} & f_{\iota,\kappa} & f_{\iota, o} \\
f_{\kappa,\iota} & f_{\kappa,\kappa} & f_{\kappa,o} \\
f_{o,\iota} & f_{o,\kappa} & f_{o,o}} 
\qquad \Longrightarrow \qquad 
H = \Matrix{f_{\kappa,\iota} & f_{\kappa,\kappa} \\
            f_{o,\iota} & f_{o,\kappa}}
\end{align}
where $J$ and $H$ are both functions of $(X(\II),\II)$ as in \eqref{eq:system}.  More precisely: 

\begin{lemma}[{\cite[Lemma 1.5]{WHAG21}}] \label{lem:irreducible}
The input-output function $x_o(\II)$ undergoes infinitesimal homeostasis at $\II_0$ if and only if $\det(H) = 0$, evaluated at $(X_0,\II_0)$. 
\end{lemma}

In \cite{WHAG21} the authors show that the determination of infinitesimal homeostasis in an input-output networks reduces to the study of core networks. We assume throughout that 
the input-output networks are core networks.  See Definition~\ref{defn:io_network}.

\begin{theorem}[{\cite[Theorem 1.11]{WHAG21}}] \label{thm:irreducible}
Assume \eqref{eq:system} has a hyperbolic equilibrium at $(X_0,\II_0)$. Then there are permutation matrices $P$ and $Q$ such that $PHQ$ is block upper triangular with square diagonal blocks $B_1,\ldots,B_m$. The blocks $B_j$ are irreducible in the sense that each $B_j$ cannot be further block triangularized. It follows that 
\begin{align}\label{eq:irreducible}
    \det(H) = \det(B_1)\cdots\det(B_m)
\end{align}
is an irreducible factorization of $\det(H)$. 
\end{theorem}

\begin{definition} \rm
Let $\mathcal{G}$ be an input-output network and $H$ its homeostasis matrix. Each irreducible 
square block $B_{\eta}$ in \eqref{eq:irreducible} is called a {\em homeostasis block}.  Further we 
say that infinitesimal homeostasis in $\GG$ is of homeostasis type $B_{\eta}$ if for all $\xi \neq \eta$
\begin{equation}
\det (B_{\eta}) = 0 \ \text{ and } \ \det (B_{\xi}) \neq 0.
\end{equation}
\end{definition}

\begin{remark}[{\cite[Section~1.10]{WHAG21}}] \rm 
Let $B_\eta$ be a homeostasis type and let
\[
h_\eta(\II) \equiv \det B_\eta(X(\II),\II)
\]
A chair point of type 
$\eta$ occurs at $\II_0$ if $h_\eta(\II_0) = h_\eta '(\II_0) = 0$ and $h_\eta '' (\II_0) \neq 0$. \END
\end{remark}

In principle every homeostasis type can lead to infinitesimal homeostasis, that is, $h_\eta(\II_0) = 0$ for some input 
value $\II_0$. For simplicity, we say that node $x_o$ is \emph{homeostatic at $\II_0$}.   We ask: 
If the output node is homeostatic  at $\II_0$, which other nodes must also be homeostatic at $\II_0$?
Based on this question we introduce the following concept:

\begin{definition} \rm
A {\em homeostasis pattern} corresponding to the homeostasis block $B_\eta$ at $\II_0$ is the collection of 
all nodes, including the output node $o$, that are simultaneously forced to be homeostatic at $\mathcal{I}_0$.\END
\end{definition}

\subsection{Example of Direct Calculation of Homeostasis Patterns}
\label{intro:example}

Finally, we determine the four homeostasis patterns by direct calculation.  We do this by  assuming that there is a one-parameter 
family of stable equilibria $X(\II)$ where $X(\II_0) = X_0$  using implicit differentiation with respect to $\II$
(indicated by $'$), and expanding \eqref{e:admissible} to first order at  $\II_0$.  The linearized system of equations is
\begin{equation}  \label{e:1st_order}
\begin{array}{rcl}
0 & = & f_{\iota,\iota}\iota' + f_{\iota,\tau_1}\tau_1' + f_{\iota,\II} \\
0 & = & f_{\sigma,\iota}\iota' + f_{\sigma,\sigma}\sigma' + f_{\sigma,\tau_2}\tau_2' \\
0 & = & f_{\tau_1,\sigma}\sigma' + f_{\tau_1,\tau_1}\tau_1' \\
0 & = & f_{\tau_2,\tau_2}\tau_2' + f_{\tau_2,\tau_3}\tau_3' \\
0 & = & f_{\tau_3,\tau_3}\tau_3'\\
0 & = & f_{o,\iota}\iota' + f_{o,\sigma}\sigma'
\end{array}
\end{equation}
Next we compute the homeostasis patterns corresponding to the $4$ homeostasis types of \eqref{e:admissible}.

\subsubsection*{(a) Homesotasis Type: $f_{\tau_3,\tau_3} = 0$. Homeostatic nodes: $\{o\}$}  
Equation~\eqref{e:1st_order} becomes
\begin{equation}  \label{e:1st_orderA}
\begin{array}{rcl}
0 & = & f_{\iota,\iota}\iota' + f_{\iota,\tau_1}\tau_1' + f_{\iota,\II} \\
0 & = & f_{\sigma,\iota}\iota' + f_{\sigma,\sigma}\sigma' + f_{\sigma,\tau_2}\tau_2' \\
0 & = & f_{\tau_1,\sigma}\sigma' + f_{\tau_1,\tau_1}\tau_1' \\
0 & = & f_{\tau_2,\tau_2}\tau_2' + f_{\tau_2,\tau_3}\tau_3' \\
0 & = & 0 \\
0 & = & f_{o,\iota}\iota' + f_{o,\sigma}\sigma'
\end{array}
\end{equation}
Since $f_{\tau_2,\tau_2}$ and $f_{\tau_2,\tau_3}$ are generically nonzero at homeostasis, the 
fourth equation implies that generically $\tau_2'$ and $\tau_3'$ are nonzero.
The second and sixth equations can be rewritten as
\[
\Matrixc{f_{\sigma,\iota} &  f_{\sigma,\sigma} \\  f_{o,\iota} & f_{o,\sigma}}
\Matrixc{\iota'  \\ \sigma'} = -\Matrixc{f_{\sigma,\tau_2}\tau_2' \\ 0}
\]
Generically the right hand side of this matrix equation at $\II_0$ is nonzero; hence generically $\iota'$ and $\sigma'$ are also nonzero.  The third equation implies that generically $\tau_1'$ is nonzero.  Therefore, in this case, the only homeostatic node is $o$.

\subsubsection*{(b) Homeostasis Type: $f_{\tau_1,\tau_1} = 0$. Homeostatic nodes: $\{ \iota, \tau_2, \tau_3, \sigma, o\}$}  

In this case \eqref{e:1st_order} becomes
\begin{equation}  \label{e:1st_orderB}
\begin{array}{rcl}
0 & = & f_{\iota,\iota}\iota' + f_{\iota,\tau_1}\tau_1' + f_{\iota,\II} \\
0 & = & f_{\sigma,\iota}\iota' + f_{\sigma,\sigma}\sigma' + f_{\sigma,\tau_2}\tau_2' \\
0 & = & f_{\tau_1,\sigma}\sigma'  \\
0 & = & f_{\tau_2,\tau_2}\tau_2' + f_{\tau_2,\tau_3}\tau_3' \\
0 & = & f_{\tau_3,\tau_3}\tau_3'\\
0 & = & f_{o,\iota}\iota' + f_{o,\sigma}\sigma'
\end{array}
\end{equation}
The fifth equation implies that generically $\tau_3' = 0$.  The fourth equation implies that generically $\tau_2'=0$. The third equation implies that generically $\sigma' = 0$ and the sixth equation implies that generically $\iota'$ is zero.
 It follows that the infinitesimal homeostasis pattern is  $\iota' =\tau_2' = \tau_3' = \sigma'  = o' = 0$.

\subsubsection*{(c) Homeostasis Type: $f_{\tau_2,\tau_2} = 0$. Homeostatic nodes: $\{\tau_3,o\}$}  

Equation~\eqref{e:1st_order} becomes
\begin{equation}  \label{e:1st_orderC}
\begin{array}{rcl}
0 & = & f_{\iota,\iota}\iota' + f_{\iota,\tau_1}\tau_1' + f_{\iota,\II} \\
0 & = & f_{\sigma,\iota}\iota' + f_{\sigma,\sigma}\sigma' + f_{\sigma,\tau_2}\tau_2' \\
0 & = & f_{\tau_1,\sigma}\sigma' + f_{\tau_1,\tau_1}\tau_1' \\
0 & = &  f_{\tau_2,\tau_3}\tau_3' \\
0 & = & f_{\tau_3,\tau_3}\tau_3'\\
0 & = & f_{o,\iota}\iota' + f_{o,\sigma}\sigma'
\end{array}
\end{equation}
The fourth or fifth equation implies that $\tau_3' = 0$. The first and sixth equations imply that $\iota'$, $\sigma'$, and $\tau_2'$  are nonzero. The third equation implies that generically $\tau_1'$ is nonzero.  Hence the 
infinitesimal homeostasis pattern is $\{\tau_3,o\}$.

\subsubsection*{(d) Homeostasis Type: $f_{\sigma,\iota} f_{o,\sigma} - f_{\sigma,\sigma} f_{o,\iota} = 0$. Homeostatic 
nodes $\{\tau_2, \tau_3, o\}$.} 

To repeat, Equation \eqref{e:1st_order} is
\begin{equation}  \label{e:1st_orderD}
\begin{array}{rcl}
0 & = & f_{\iota,\iota}\iota' + f_{\iota,\tau_1}\tau_1' + f_{\iota,\II} \\
0 & = & f_{\sigma,\iota}\iota' + f_{\sigma,\sigma}\sigma' + f_{\sigma,\tau_2}\tau_2' \\
0 & = & f_{\tau_1,\sigma}\sigma' + f_{\tau_1,\tau_1}\tau_1' \\
0 & = & f_{\tau_2,\tau_2}\tau_2' + f_{\tau_2,\tau_3}\tau_3' \\
0 & = & f_{\tau_3,\tau_3}\tau_3'\\
0 & = & f_{o,\iota}\iota' + f_{o,\sigma}\sigma'
\end{array}
\end{equation}
The fifth equation implies generically that $\tau_3' = 0$ and the fourth equation implies generically that $\tau_2' = 0$. Again, the second and sixth equations can be rewritten in matrix form as
\[
\Matrixc{f_{\sigma,\iota} &  f_{\sigma,\sigma} \\  f_{o,\iota} & f_{o,\sigma}}
\Matrixc{\iota'  \\ \sigma'} = -\Matrixc{f_{\sigma,\tau_2}\tau_2' \\ 0} = 0
\]
Hence generically $\iota'$ and $\sigma'$ are nonzero. The third equation implies that $\tau_1$ is nonzero.  Hence the homeostatic nodes are $\tau_2,\tau_3,o$.

The homeostasis types of \eqref{e:admissible} with the corresponding homeostasis patterns are summarized in Table~\ref{T:example8}.

\begin{table}[!htb]
\begin{center}
\begin{tabular}{|c|l|l|c|}
\hline
Homeostasis Type & Homeostasis Pattern &  Figure~\ref{F:admissible} \\
\hline
$f_{\tau_3,\tau_3} = 0$ & $\{o\}$  &(a)  \\
$f_{\tau_1,\tau_1} = 0$ & $\{\iota,\tau_2,\tau_3,\sigma,o\}$ & (b)\\
$f_{\tau_2,\tau_2} = 0$ & $\{\tau_3,o\}$ & (c) \\
$f_{\sigma,\iota} f_{o,\sigma} - f_{\sigma,\sigma} f_{o,\iota} = 0$ & $\{\tau_2,\tau_3,o\}$ & (d) \\
\hline
\end{tabular}
\end{center}
\caption{Infinitesimal homeostasis patterns for admissible systems in \eqref{e:admissible}}
\label{T:example8} 
\end{table}

In principle the homeostasis patterns of any input-output network can be computed in the manner shown above, but in practice this becomes complicated for large networks.
In this paper we introduce another approach based on the pattern network $\PP$ associated to the input-output network $\GG$. 
This method is both computationally and theoretically superior.  
First, the method introduced here provides a reduction of the size of the original input-output network to the pattern network, which is obtained from the former by `collapsing' certain subsets of nodes into single nodes.
Second, the classification of the homeostasis patterns using the pattern network is given by an algorithm (which can be easily extracted from the main theorems).
We illustrate this by working out the homeostasis patterns of the network shown in Figure \ref{T:example8} using the new approach in Example \ref{E:example8_pattern}.
Finally, the new conceptual framework allows us to give a new characterization of the homeostasis types, namely, that they correspond uniquely to the homeostasis patterns.

\subsection{Structure of the Paper}
\label{intro:structure}

Sections \ref{S:homeostasis}-\ref{S:homeo_induce}, introduce the terminology of input-output networks, the homeostasis pattern network $\PP$, and homeostasis induction.  In Section \ref{SS:CHP} we state four of the main theorems of this paper, Theorems \ref{thm:struct_to_struct}, \ref{thm:struct_to_app}, \ref{thm:app_to_struct}, \ref{thm:app_to_app}, which characterize homeostasis patterns combinatorially.
Section \ref{sec:overview} provides an overview of the proofs of these main results.   In 
Section \ref{sec:comb_charact}, we consider combinatorial characterizations of the input-output 
networks $\GG(\KK)$ that are obtained by repositioning the output node on a given input-output network from $o$ to $\kappa$.  Sections \ref{sec:struct_pattern} and \ref{sec:app_pattern} determine the structural and appendage homeostasis pattern respectively. In Section \ref{S:properties_induction} we discuss properties of homeostasis induction. Specifically we show that a homeostasis pattern uniquely determines its homeostasis type. This result is a restatement of Theorem~\ref{thm:bijective_corresp_intro}.  

The paper ends in Section~\ref{S:discussion} with a brief discussion of various types of networks  that support different aspects of infinitesimal homeostasis.  These networks include gene regulatory networks (GRN), input-output networks with input node = output node, and higher codimension types of infinitesimal homeostasis.

\section{Aspects of Input-Output Networks}
\label{terminology}

In this section, we recall additional basic terminology and results on infinitesimal homeostasis in input-output networks from~\cite{WHAG21}. 

\subsection{Homeostasis Subnetworks}
\label{S:homeostasis}

Wang \etal~\cite[Definition 1.14]{WHAG21} associate a \emph{homeostasis subnetwork} $\KK_\eta\subset\GG$ with each homeostasis block $B_\eta$  (recall Theorem \ref{thm:irreducible}) and give a graph-theoretic description of each $\KK_\eta$.
More precisely, the subnetwork $\KK_\eta$ of $\GG$ associated with the homeostasis block $B_\eta$ is
defined as follows. The nodes in $\KK_\eta$ are the union of nodes $p$ and $q$ where $f_{p,x_q}$ is
a nonzero entry in $B_\eta$ and the arrows of $\KK_\eta$ are the union of arrows $q \to p$ where
$p \neq q$.

\begin{definition}[Definition 1.15 of \cite{WHAG21}] \label{D:simple_node}
Let $\GG$ be a core input-output network. 
\begin{enumerate}[(a)]
\item A \emph{simple path} from node $\kappa_1$ to node $\kappa_2$ in $\GG$ is a directed path that starts at $\kappa_1$, ends at $\kappa_2$, and visits each node on the path exactly once. We denote the existence of a 
simple path from $\kappa_1$ to $\kappa_2$ by $\kappa_1 \pathto \kappa_2$. 
A \emph{simple cycle} is a simple path whose first and last nodes are identical. \item An \emph{$\iota o$-simple path} is a simple path from the input node $\iota$ to the output node $o$. 

\item A node $\sigma$ is \emph{simple} if it lies on an $\iota o$-simple path. A node $\tau$ is \emph{appendage} if it is not simple.
\item A simple node $\rho$ is \emph{super-simple} if it lies on every $\iota o$-simple path. \END
\end{enumerate}
\end{definition}   

We typically use $\sigma$ to denote a simple node, $\rho$ to denote a super-simple node, and $\tau$ to denote an appendage node when the type of the node is assumed \textit{a priori}. Otherwise, we use $\kappa$ to denote an arbitrary node. Note that $\iota$ and $o$ are super-simple nodes.  

Let $\rho_0, \rho_1, \ldots, \rho_q, \rho_{q+1}$ be the super-simple nodes, where $\rho_0 = \iota$ and $\rho_{q+1} = o$.
The super-simple nodes are totally ordered by the order of their appearance on any $\iota o$-simple path, and this ordering is independent of the $\iota o$-simple path.  We denote
an $\iota o$-simple path by 
\[
\iota \pathto \rho_1 \pathto \cdots \pathto \rho_q \pathto o  
\]
where $\rho_j\pathto\rho_{j+1}$ indicates a simple path from $\rho_j$ to $\rho_{j+1}$. 
The ordering of the super-simple nodes is denoted by  
\[
    \rho_0 \prec \rho_1 \prec \cdots \prec \rho_q \prec \rho_{q+1},
\]
and $\prec$ is a total ordering. 
The ordering $\prec$ extends to a partial ordering of simple nodes, as follows. 
If there exists a super-simple node $\rho$ and an $\iota o$-simple path such that 
\[
    \iota \pathto \sigma_1 \pathto \rho \pathto \sigma_2 \pathto o
\]
then the partial orderings 
\[
    \sigma_1 \prec \rho \qquad \rho \prec \sigma_2 \qquad \sigma_1 \prec \sigma_2
\]
are valid. In this partial ordering every simple node is comparable to every super-simple node 
but two simple nodes that lie between the same adjacent super-simple nodes need not be 
comparable.

We recall the definition of transitive (or strong) components of a network. Two nodes are {\em equivalent} if there is a path from one to the other and back. A {\em transitive component} is an equivalence class for this equivalence relation.

\begin{definition}\label{defn:super_appendage}
\begin{enumerate}[(a)]
\item Let $S$ be an $\iota o$-simple path.  The {\em complementary subnetwork} of $S$ is the network $C_S$ whose nodes are nodes that are not in $S$ and whose arrows are those that connect nodes in $C_S$.

\item An appendage node $\tau$ is {\em super-appendage} if for each $C_S$  containing $\tau$, the transitive component of $\tau$ in $C_S$ consists only of appendage nodes. \END
\end{enumerate}
\end{definition} 

Note that this definition of super-appendage leads to a slightly different, but equivalent, definition of homeostasis subnetwork to the one given in \cite{WHAG21}. 
However, this change enables us to define pattern networks in a more straightforward way (see Remark \ref{R:appendage}).

Now we start with the definition of structural subnetworks. 

\begin{definition} \label{D:structural_subnet}
Let $1\leq  j \leq q+1$ and $\rho_{j-1}\prec\rho_j$ be two consecutive super-simple nodes.
Then the $j^{th}$ \emph{simple subnetwork} $\LL_j''$, 
the $j^{th}$ {\em augmented simple subnetwork} $\LL_j'$, and
the $j^{th}$ \emph{structural subnetwork} $\LL_j$ 
are defined in four steps as follows.
\begin{enumerate}[(a)]
\item The $j^{th}$ \emph{simple subnetwork} $\LL_j''$ consists of simple nodes $\sigma$ where 
\[
    \rho_{j-1}\prec \sigma \prec \rho_j
\] 
and all arrows connecting these nodes. Note that $\LL_j''$ does not contain the super-simple nodes $\rho_{j-1}$ and $\rho_j$, and  $\LL_j''$ can be the empty set.

\item An appendage but not super-appendage node $\tau$ is {\em linked} to $\LL_j''$ if for some complementary subnetwork $C_S$ the transitive component of $\tau$ in $C_S$ is the union of $\tau$, nodes in $\LL_j''$, and non-super-appendage nodes. The set of $j^{th}$-{\em linked  appendage nodes} $T_j$ is the set of non-super-appendage nodes that are linked to $\LL_j''$.

\item The $j^{th}$ {\em augmented simple subnetwork} $\LL_j'$ is    
\[
 	\LL_j' =   \LL_j'' \cup T_j
\]
and all arrows connecting these nodes. 
\item The $j^{th}$ \emph{structural subnetwork} $\LL_j$ consists of the augmented simple subnetwork $\til{\LL}_j$ and adjacent super-simple nodes. That is 
\[
    \LL_j = \{\rho_{j-1}\} \cup \LL_j' \cup \{\rho_j\}
\]
and all arrows connecting these nodes. \END
\end{enumerate}
\end{definition}

\begin{definition}
Define $\sigma_1 \preceq \sigma_2$ if either $\sigma_1 \prec \sigma_2$, $\sigma_1 = \sigma_2$ 
is a super-simple node, or $\sigma_1$ and $\sigma_2$ are in the same simple subnetwork. \END
\end{definition}

Next we define the appendage subnetworks, which were defined in Section 1.7.2 of \cite{WHAG21} as any transitive component of the subnetwork consisting only of appendage nodes and the arrows between them.  

\begin{definition} \label{D:appendage_subnet}
An \emph{appendage subnetwork} $\AA$ is a transitive component of the subnetwork of super-appendage nodes. \END
\end{definition}

\begin{remark} \rm \label{R:appendage}
Wang \etal~\cite{WHAG21} define an appendage subnetwork as a transitive component of appendage nodes $\AA$ that satisfy the {\em `no cycle' condition}.
This condition is formulated in terms of the non-existence of a cycle between appendage nodes in $\AA$ and 
the simple nodes in $C_S$ for all simple $\iota o$-simple paths $S$.
Here, we define an appendage subnetwork as a transitive component of super-appendage nodes, which are defined in terms of transitive components with respect to $C_S$ for all simple $\iota o$-simple paths $S$. These two definitions are equivalent because two nodes belong to the same transitive component if and only if both nodes lie on a (simple) cycle.  \END
\end{remark}

Wang~\etal~\cite{WHAG21} show that each {\em homeostasis subnetwork} of $\GG$ is either structural 
(satisfies Definition \ref{D:structural_subnet}(d)) or  appendage (satisfies Definition~\ref{D:appendage_subnet}).

\subsection{Homeostasis Pattern Network}
\label{intro:pattern}

In this subsection we construct the \emph{homeostasis pattern network} $\PP$ associated with $\GG$ (see
Definition~\ref{D:HPN}), which serves to organize the homeostasis subnetworks and to clarify how each homeostasis subnetwork connects to the others.

The homeostasis pattern network is defined in the following steps.
First, we define the structural pattern network $\PP_\sS$ in terms of the structural subnetworks of $\GG$.
Second, we define the appendage pattern network $\PP_\AA$ in terms of the appendage subnetworks of $\GG$.
Finally, we define how the nodes in $\PP_\sS$ connect to nodes in $\PP_\AA$ and conversely.

\begin{definition}
The {\em structural pattern network} $\PP_\sS$ is the feedforward network whose nodes are 
the super-simple nodes $\rho_j$ and the \emph{backbone} nodes $\til{\LL}_j$, where $\til{\LL}_j$ is  
the augmented structural subnetwork $\LL_j'$ treated as a  single node.  The nodes and arrows of 
$\PP_\sS$ are given as follows.
\begin{equation} \label{backbone_e}
\iota = \rho_0 \to \til{\LL}_1 \to \rho_1 \to \til{\LL}_2 \to \cdots  \to \til{\LL}_{q+1} \to  \rho_{q+1} = o 
\end{equation}
\end{definition}

If a structural subnetwork ${\LL}$ consists of an arrow between two adjacent super-simple nodes (Haldane homeostasis type) then the corresponding augmented structural subnetwork ${\LL'}$ is the empty network; nevertheless the corresponding backbone node $\til{\LL}$ must be included in the structural pattern network $\PP_\sS$.

\begin{definition}
The {\em appendage pattern network} $\PP_\AA$ is the network whose nodes are the components $\til{\AA}$ in the condensation of the subnetwork of super-appendage nodes. 
Such a node $\til{\AA}$ is called an {\em appendage component}.
An arrow connects nodes $\til{\AA}_1$ and $\til{\AA}_2$ if and only if there are super-appendage nodes $\tau_1 \in \til{\AA}_1$ and $\tau_2 \in \til{\AA}_2$ such that $\tau_1 \to \tau_2$ in $\GG$. 
\END
\end{definition}

The {\em condensation} $\GG^c$ of a network $\GG$ is defined as follows. 
The vertices of $\GG^c$ are strong components (or transitive components) of $\GG$, and the edge in $\GG^c$ is present only if there exists at least one edge between the vertices of corresponding connected components.

To complete the homeostasis pattern network, we describe how the nodes in $\PP_\AA$  and the nodes in 
$\PP_\sS$ are 
connected. To do so, we take advantage of the feedforward ordering of the nodes in $\PP_\sS$ and the feedback ordering of the nodes in $\PP_\AA$. 

\begin{definition} 
A simple path from $\kappa_1$ to $\kappa_2$ is an \emph{appendage path} if some node on this path is an appendage node and every node on this path, except perhaps for $\kappa_1$ and $\kappa_2$, is an 
appendage node. \END
\end{definition}

\paragraph{How $\PP_\AA$ connects to $\PP_\sS$.} 
\begin{definition} \label{D:V_max}
Given a node $\til{\AA}\in\PP_\AA$, we construct a unique arrow from $\til{\AA}$ to the structural pattern 
network $\PP_\sS$ in two steps: 
\begin{enumerate}[(a)]
\item Consider the collection of nodes $\VV$ in $\PP_\sS$ for 
which there exists a simple node $\sigma \in \VV$ and appendage node $\tau \in \til{\AA}$, such that there is an  appendage path from $\tau$ to $\sigma$.

\item Let $\VV_{max}(\til\AA)$ be a maximal node in this collection, that is, the most downstream in $\PP_\sS$. It follows from \eqref{backbone_e} that $\VV_{max}$ is either a super-simple node $\rho_j$ or a backbone node $\til{\LL}_j$. Maximality implies that $\VV_{max}$ is uniquely defined. We then say that there is an arrow from $\til{\AA}$ to $\VV_{max} \in \PP_\sS$. \END
\end{enumerate}
\end{definition}

\paragraph{How $\PP_\AA$ is connected from $\PP_\sS$.} 
\begin{definition} \label{D:V_min}
Given a node $\til{\AA}\in\PP_\AA$ we choose uniquely an arrow from the structural pattern 
network $\PP_\sS$ to $\til{\AA}$ in two steps: 

\begin{enumerate}[(a)]
\item Consider the collection of nodes $\VV$ in $\PP_\sS$ 
for which there exists a simple node $\sigma \in \VV$ and appendage node $\tau \in \til{\AA}$, such that there is an  appendage path from $\sigma$ to $\tau$.

\item Let $\VV_{min}(\til{\AA})$ be a minimal node in this collection, that is, the most upstream node in $\PP_\sS$. Then $\VV_{min}$ is either a super-simple node $\rho_j$ or a backbone node $\til{\LL}_j$, and the minimality implies  uniqueness of $\VV_{min}$. We then say that there is an arrow from $\VV_{min} \in \PP_\sS$ to $\til{\AA}$. \END
\end{enumerate}
\end{definition}

Since we consider only core input-output networks, all appendage nodes are downstream from $\iota$ and upstream from $o$.
Hence, for any node $\til{\AA}\in\PP_\AA$, there always exist nodes $\VV_{min}, \VV_{max} \in \PP_\sS$ as mentioned above.

\begin{definition} \label{D:HPN}
The \emph{homeostasis pattern network} $\PP$ is the network whose nodes are the union of the nodes of the structural pattern network $\PP_\sS$ and the appendage pattern network $\PP_\AA$. The arrows of $\PP$ are the arrows of $\PP_\sS$, the arrows of $\PP_\AA$, and the arrows between $\PP_\sS$ and $\PP_\AA$ as described above. \END
\end{definition}

\begin{remark} \rm \label{R:ss_nocorrespond}
Note that the super-simple nodes in $\PP$  correspond to the super-simple nodes of $\GG$. 
Each super-simple node $\rho_j \in \GG$ (for $1\le j\le q$) belongs to exactly two structural subnetworks $\LL_{j-1}$ and $\LL_{j}$.  Thus they are not associated to a single homeostasis subnetwork of $\GG$.  \END
\end{remark}

It follows from Remark~\ref{R:ss_nocorrespond} that there is a correspondence between the homeostasis subnetworks of $\GG$ and the non-super-simple nodes of $\PP$.

\begin{remark} \label{R:correspond} \rm $ $
\begin{enumerate}[(a)]
\item Each structural subnetwork $\LL \subseteq \GG$  corresponds to the backbone node $\til{\LL} \in \PP_\sS$. Note that the augmented structural subnetworks $\LL' \subsetneq \LL$ are not homeostasis subnetworks.

\item Each appendage subnetworks $\AA \subset \GG$ corresponds to a appendage component $\til{\AA} \in \PP_\AA$.

\item For simplicity in notation we let $\VV_\sS$ denote a node in $\PP_\sS$.
Further we let $\til{\VV}$ denote a non-super-simple node of $\PP$ and $\VV$ denote its corresponding homeostasis subnetwork.
\END
\end{enumerate}
\end{remark}

\subsection{Homeostasis Induction}
\label{S:homeo_induce}

Here we define \emph{homeostasis induction} in the homeostasis pattern network $\PP$, which is critical to determining the homeostasis pattern `triggered' by each homeostasis subnetwork.

\begin{definition} \label{D:homeo_inducing}
Assume that the output node $o$ is homeostatic at  $(X_0,\II_0)$, that is, $x_o'(\II_0) = 0$ for some input value $\II_0$.

\begin{enumerate}[(a)]
\item We call the homeostasis subnetwork 
$\KK_\eta$  \emph{homeostasis inducing} if $h_{\KK_\eta} \equiv \det(B_\eta) = 0$ at $(X_0,\II_0)$. 

\item Homeostasis of a node $\kappa \in \GG$ is \emph{induced} by a homeostasis subnetwork $\KK$, denoted $\KK \Rightarrow \kappa$, if generically for $f$ in \eqref{eq:system}  $\kappa$ is homeostatic whenever $\KK$ is homeostasis inducing.

\item A homeostasis subnetwork $\KK$ \emph{induces} a subset of nodes $\NN$ ($\KK \Rightarrow \NN$), if $\KK \Rightarrow \kappa$ for each node $\kappa \in \NN \subset \GG$. \END
\end{enumerate}
\end{definition}

By definition, every homeostasis subnetwork $\KK$ induces homeostasis in the output node $o$, that is, $\KK \Rightarrow o$.

The main point of introducing the homeostasis pattern network $\PP$ is to
relate homeostatic induction between the set of homeostasis subnetworks of $\GG$ to induction between nodes in $\PP$.
In Definition \ref{D:homeo_inducing_pattern} bellow  we formalize this notion.
Hence, every node in a homeostasis pattern (which can be backbone or appendage) is induced by either a backbone node or an appendage node in the homeostasis pattern network $\PP$.

\begin{definition} \label{D:homeo_inducing_pattern}
Let $\til{\VV}_1, \til{\VV}_2 \in \PP$ be non-super-simple nodes and $\rho\in\PP$ be a super-simple node.
Let $\VV_1, \VV_2 \subset \GG$ be the corresponding homeostasis subnetworks to $\til{\VV}_1, \til{\VV}_2 \in \PP$.
We say that $\til{\VV}_1$ {\em induces} $\til{\VV}_2$, denoted by $\til{\VV}_1 \Rightarrow \til{\VV}_2$, if and only if $\VV_1 \Rightarrow \VV_2$. 
We say that $\til{\VV}_1$ {\em induces} $\rho$, denoted by $\til{\VV}_1 \Rightarrow \rho$, if and only if $\VV_1 \Rightarrow \rho$. \END
\end{definition}

We exclude super-simple nodes of $\PP$ from being `homeostasis inducing' because they are not associated to a homeostasis subnetwork of $\GG$ (see Remark \ref{R:ss_nocorrespond}).
However, when a backbone node $\til{\LL}_j\in \PP$ induces homeostasis on other nodes of $\PP$, it is the corresponding structural subnetwork $\LL_j$, with its two super-simple nodes $\rho_{j-1}, \rho_j$ that induce homeostasis.

\subsection{Characterization of Homeostasis  Patterns} 
\label{SS:CHP}

As explained before, the homeostasis pattern network $\PP$ allows us to characterize homeostasis patterns by reducing to four possibilities that are covered by Theorems \ref{thm:struct_to_struct} - \ref{thm:app_to_app}.

Structural homeostasis patterns are given by the following two theorems.

\begin{theorem}[Structural Homeostasis $\Rightarrow$ Structural Subnetworks]\label{thm:struct_to_struct}
A backbone node $\til{\LL}_j \in \PP_\sS$ induces
every node of the structural pattern network $\PP_S$ strictly downstream from $\til{\LL}_j$, but no other nodes of $\PP_S$.
\end{theorem}

See Figure \ref{fig:struct_to_struct} for an application of Theorem~\ref{thm:struct_to_struct}. 

\begin{figure}[H]
\centering
\includegraphics[width = .8\textwidth]{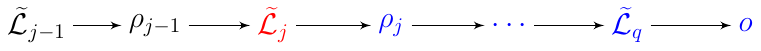}
\caption{An example of structural network induced by structural homeostasis. 
Suppose the backbone node $\til{\LL}_j$ in red is homeostasis inducing.  Then Theorem \ref{thm:struct_to_struct} 
implies that the blue nodes in the structural pattern network are all homeostatic.
}
\label{fig:struct_to_struct}
\end{figure}

\begin{theorem}[Structural Homeostasis $\Rightarrow$ Appendage Subnetworks]\label{thm:struct_to_app}
A backbone node $\til{\LL}_j \in \PP_\sS$ induces every
appendage component of $\PP_\AA$ whose $\VV_{min}$ (see Definition~\ref{D:V_min}) is strictly downstream, but no other nodes of $\PP_\AA$.
\end{theorem}

See Figure  \ref{fig:struct_to_app} for an application of Theorem~\ref{thm:struct_to_app}.  

\begin{figure}[!htb]
\centering
\includegraphics[width = .33\textwidth]{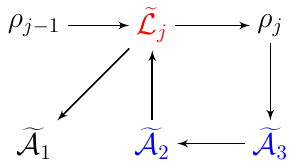}
\caption{An example of appendage subnetworks induced by structural homeostasis.  
Suppose the backbone node $\til{\LL}_j$ in red is homeostasis inducing. Then by Theorem \ref{thm:struct_to_app} and the fact that the super-simple node $\rho_j$ is strictly downstream from $\til{\LL}_j$ the blue appendage components downstream from $\rho_j$ are homeostatic.}
\label{fig:struct_to_app}
\end{figure}

The appendage homeostasis patterns are characterized by the following two theorems.

\begin{theorem}[Appendage Homeostasis $\Rightarrow$ Structural Subnetworks] \label{thm:app_to_struct}
An appendage component $\til{\AA} \in \PP_\AA$ induces every super-simple node of $\PP_S$ downstream 
from $\VV_{max}(\til{\AA})$ (see Definition~\ref{D:V_max}), but no other super-simple nodes. 
Further, an appendage component $\til{\AA} \in \PP_\AA$ induces a backbone node $\til{\LL}_j$ if and only if 
 $\til{\LL}_j$ is strictly downstream from $\VV_{max}(\til{\AA})$.
\end{theorem}

See Figure \ref{fig:app_to_struct} for an application of Theorem~\ref{thm:struct_to_struct}. 

\begin{figure}[!htb]
\centering
\includegraphics[width = .5\textwidth]{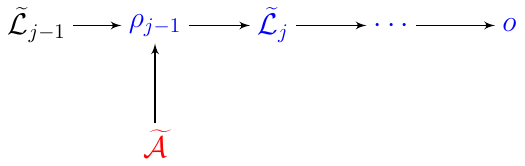}
\caption{An example of structural subnetworks induced by appendage homeostasis.
Suppose the appendage component $\til{\AA}$ in red is homeostasis inducing. 
Since $\til{\AA}$ connects to the super-simple node $\rho_{j-1}$, then by Theorem \ref{thm:app_to_struct} the blue nodes in the structural pattern network are homeostatic.}
\label{fig:app_to_struct}
\end{figure}

Please recall Definitions~\ref{D:V_max} and \ref{D:V_min} before reading the next theorem.

\begin{theorem}[Appendage Homeostasis $\Rightarrow$ Appendage Subnetworks]\label{thm:app_to_app}
An appendage component $\til{\AA}_i \in \PP_\AA$ induces an appendage component $\til{\AA}_j \in \PP_\AA$ if and only if 
$\til{\AA}_i$ is strictly upstream from $\til{\AA}_j$ and every path from $\til{\AA}_i$ to $\til{\AA}_j$ in $\PP$ 
 contains a super-simple node $\rho$ satisfying $\VV_{max} (\til{\AA}_i) \preceq \rho \preceq \VV_{min} (\til{\AA}_j)$.   
\end{theorem}

See Figure \ref{fig:app_to_app} for an application of Theorem~\ref{thm:app_to_app}. 
\begin{figure}[!htb] 
\centering
\includegraphics[width=.45\textwidth]{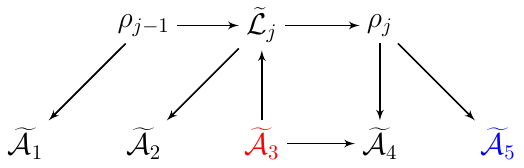}
\caption{
An example of appendage subnetworks induced by appendage homeostasis. 
Suppose the appendage component $\til{\AA}_3$ in red is homeostasis inducing. 
Since $\til{\AA}_3$ has only one path to the blue appendage component $\til{\AA}_5$ containing the super-simple node $\rho_j$, then by Theorem \ref{thm:app_to_app} $\til{\AA}_5$ is homeostatic,  but no other appendage subnetwork is homeostatic.}
\label{fig:app_to_app}
\end{figure}

\begin{example} \em \label{E:example8_pattern}

We consider homeostasis patterns for the admissible systems in \eqref{e:admissible} obtained from Figure~\ref{F:example8}.  Specifically, we show how the theorems in this section lead to the determination of the homeostasis patterns that were derived by direct calculation from the equations \eqref{e:1st_order}. 
The corresponding homeostasis pattern network $\PP$ is shown in Figure \ref{F:example8b}.
The answer is listed in Table~\ref{T:example8}.

\medskip

\noindent
\textbf{Case (a): }  $f_{\tau_3, \tau_3} = 0$ at $\II_0$:
homeostasis is induced by the node $\til{\AA}_3$ of $\PP_{\AA}$. \\
Theorem \ref{thm:app_to_struct} shows 
$\til{\AA}_3$ induces $\{ o \}$, which is the only super-simple node of $\PP_\sS$ downstream from $\VV_{max}(\til{\AA}_3)$.
And $\til{\AA}_3$ induces no backbone node.
Theorem \ref{thm:app_to_app} shows 
$\til{\AA}_3$ induces no appendage component.
Therefore, in this case the homeostasis pattern is $\{ o \}$.

\medskip

\noindent
\textbf{Case (b): }  $f_{\tau_1, \tau_1} = 0$ at $\II_0$:
homeostasis is induced by the node $\til{\AA}_1$ of $\PP_{\AA}$. \\
Theorem \ref{thm:app_to_struct} shows 
$\til{\AA}_1$ induces $\{ \iota, o \}$, which are the super-simple nodes of $\PP_\sS$ downstream from $\VV_{max}(\til{\AA}_1)$.
Also $\til{\AA}_1$ induces $\{ \til{\LL}_1 \}$, which is the backbone node of $\PP_\sS$ downstream from $\iota$.
Theorem \ref{thm:app_to_app} shows 
$\til{\AA}_1$ induces $\{ \til{\AA}_2, \til{\AA}_3 \}$, which are the nodes of $\PP_\AA$ downstream from $\til{\AA}_1$ and each path contains the super-simple node $o$ with $\VV_{max}(\til{\AA}_1) \preceq o \preceq \VV_{min}(\til{\AA}_2)$ or $\VV_{min}(\til{\AA}_3)$. Therefore, in this case the homeostasis pattern is $\{ \iota, \ \sigma, \ o, \ \tau_2, \ \tau_3 \}$.

\medskip

\noindent
\textbf{Case (c): }  $f_{\tau_2, \tau_2} = 0$ at $\II_0$:
homeostasis is induced by the node $\til{\AA}_2$ of $\PP_{\AA}$. \\
Theorem \ref{thm:app_to_struct} shows 
$\til{\AA}_2$ induces $\{ o \}$, which is the super-simple nodes of $\PP_\sS$ downstream from $\VV_{max}(\til{\AA}_2)$.
And $\til{\AA}_2$ induces no backbone node.
Theorem \ref{thm:app_to_app} shows 
$\til{\AA}_2$ induces $\{ \til{\AA}_3 \}$, which is the node of $\PP_\AA$ downstream from $\til{\AA}_2$ and each path contains the super-simple node $o$ with $\VV_{max}(\til{\AA}_2) \preceq o \preceq \VV_{min}(\til{\AA}_3)$.
Therefore, in this case the homeostasis pattern is $\{ o, \ \tau_3 \}$.

\medskip

\noindent
\textbf{Case (d): }  
$\det \begin{pmatrix}
f_{\sigma, \iota} & f_{\sigma, \sigma} \\
f_{o, \iota} & f_{o, \sigma}
\end{pmatrix} = 0$ at $\II_0$:
homeostasis is induced by node $\til{\LL}_1$ of $\PP_{\sS}$. \\
Theorem \ref{thm:struct_to_struct} shows 
$\til{\LL}_1$ induces $\{ o \}$, which is the node of $\PP_\sS$ downstream from $\til{\LL}_1$.
Theorem \ref{thm:struct_to_app} shows 
$\til{\LL}_1$ induces $\{ \til{\AA}_2, \til{\AA}_3 \}$, whose $\VV_{min}$ are strictly downstream from $\til{\LL}_1$.
Therefore, in this case the homeostasis pattern is $\{ o, \ \tau_2, \ \tau_3 \}$.
\end{example}

\begin{figure}[!htb]
\centering
\includegraphics[width =.5\textwidth]{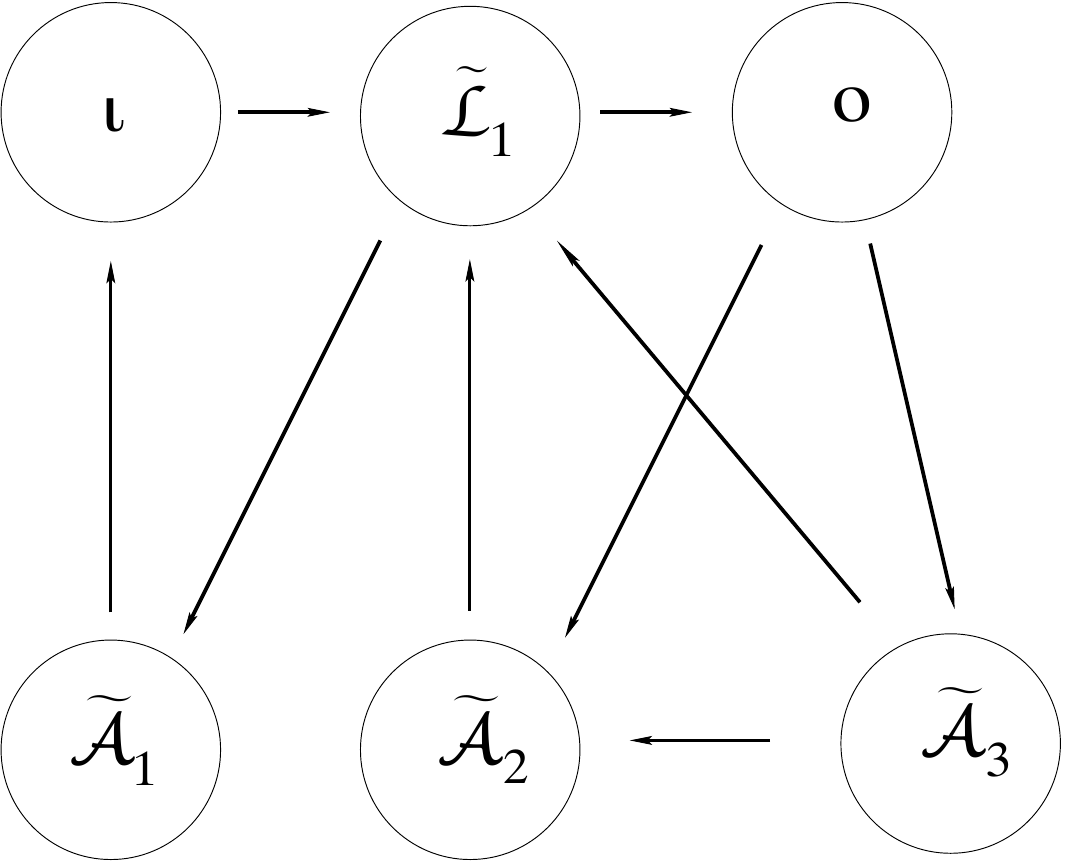}
\caption{Pattern network $\PP$ obtained from input-output network $\GG$ shown in Figure~\ref{F:example8}. \label{F:example8b}}
\end{figure}

\section{Overview of Proofs of Theorems \ref{thm:struct_to_struct}-\ref{thm:app_to_app}}
\label{sec:overview}

The theorems that we characterize here give the homeostasis patterns
$\til{\VV}_1 \Rightarrow \til{\VV}_2$ or $\til{\VV}_1 \not\Rightarrow \til{\VV}_2$ between any 
two nodes $\til{\VV}_1,\til{\VV}_2$ in the homeostasis pattern network $\PP$. Their proofs are done 
by transferring to the pattern network $\PP$ the corresponding results on the 
input-output network $\GG$. These proofs are achieved in three steps.

The first step solves the following.  Given a node $\kappa\in\GG$ and a homeostasis 
subnetwork $\KK\subset \GG$, determine whether $\KK \Rightarrow \kappa$
for each of the four possibilities:
\begin{enumerate}[(1)]
\item $\KK$ structural and $\kappa$ simple,
\item $\KK$ structural and $\kappa$ appendage,
\item $\KK$ appendage and $\kappa$ simple,
\item $\KK$ appendage and $\kappa$ appendage.
\end{enumerate}

In the first step the proof proceeds by fixing a node $\kappa\in \GG$ and considering the input-output network $\GG(\kappa)$ defined as the input-output network $\GG$ with input node $\iota$ and output node $\kappa$ (see Definition \ref{D:sigma(kappa)}).

The following result describes the role of the network $\GG(\kappa)$ 
in showing that $\KK \Rightarrow \kappa$ for some subnetwork $\KK$.

\begin{lemma} \label{L:induce_homeo}
Let $\KK$ be a homeostasis subnetwork of $\GG$.  Then $\KK$ induces $\kappa$ if and only if 
$\KK$ is a homeostasis subnetwork of $\GG(\kappa)$.
\end{lemma}

\begin{proof}
The lemma follows by recalling Definition \ref{D:homeo_inducing} of homeostasis induction
and by applying Theorem \ref{thm:irreducible} to $\GG(\kappa)$.
\end{proof}

The proofs of statements (1) - (4) above reduce to purely combinatorial statements.  Specifically
given a homeostasis subnetwork $\KK$ of $\GG$ and a node $\kappa\in\GG$, solve the following
two problems.
\paragraph{\bf Determine when $\KK$ is a structural subnetwork of $\GG(\kappa)$.} The answer is given by Lemmas \ref{lem:shared_super_simple} and \ref{lem:shared_structural} in Section \ref{S:structural_subnetworks}.

\paragraph{\bf Determine when $\KK$ is an appendage subnetwork of $\GG(\kappa)$.} The answer is given by Lemmas \ref{L:app_connect_order} - \ref{L:kappa_equal_AA} in Section \ref{S:appendage_subnetworks}.

\medskip

At this step in the proof we work directly with the input-output networks $\GG$ and $\GG(\kappa)$.

\medskip

The second step consists of lumping together the induced/non-induced nodes $\kappa$ into the corresponding homeostasis subneworks.   That is, if $\KK \Rightarrow \kappa$ then $\KK \Rightarrow \KK'$, 
where $\KK \neq \KK'$ and  
$\kappa\in\KK'$.
This is done in several propositions in Sections
\ref{sec:struct_pattern} and \ref{sec:app_pattern}.

\medskip

The third step consists of transferring the relations of induction/non-induction between homeostasis subnetworks to relations between the nodes of the homeostasis pattern network $\PP$.
Since $\PP$ is obtained by collapsing certain subsets of nodes of $\GG$, this step follows 
automatically from the previous step.

\section{Combinatorial Characterization of the Input-Output Networks $\GG(\kappa)$}
\label{sec:comb_charact}

\subsection{The Input-Output Networks $\GG(\kappa)$}

\begin{definition}  \label{D:G(kappa)}
Let $\GG$ be an input-output network with input node $\iota$ and output node $o$, and let $\kappa \in \GG$ be a node. 
Define the input-output network $\GG(\kappa)$ to be the network $\GG$ with input node $\iota$ and output node $\kappa$.  \END
\end{definition}

\begin{definition}  \label{D:sigma(kappa)}
Let $\GG$ be an input-output network 
and let $\kappa \in \GG$ be a node. Define the nodes $\sigma^u(\kappa)$ and $\sigma^d(\kappa)$ as follows. 
\begin{enumerate}[(a)]

\item If $\kappa$ is simple, then $\sigma^u(\kappa) = \kappa = \sigma^d(\kappa)$.

\item If $\kappa$ is appendage, then $\sigma^u(\kappa)$ is a minimal upstream simple node with an appendage path to $\kappa$ and $\sigma^d(\kappa)$ is a maximal downstream simple node with an appendage path from $\kappa$. \END

\end{enumerate}
\end{definition}

\begin{remark} \em 
Let $\AA$ be an appendage subnetwork.
Since $\AA$  is a transitive component, we can choose an arbitrary node 
$\tau \in \AA$ and observe that $\sigma^u(\AA) = \sigma^u(\tau)$ and $\sigma^d(\AA) =\sigma^d(\tau)$.
\END
\end{remark}

If $\kappa$ is an appendage node, then $\sigma^u(\kappa)$ is upstream from $\kappa$ and $\sigma^d(\kappa)$ is downstream from $\kappa$. 
If $\kappa$ is a node in an appendage subnetwork $\AA$, then $\sigma^u(\kappa)$ is contained in some structural subnetwork or is a super-simple node $\VV$ with an arrow $\til{\VV} \to \til{\AA}$ and $\sigma^d(\kappa)$ is contained in some structural subnetwork or is a super-simple node $\VV$ with an arrow $\til{\AA} \to \til{\VV}$.

\begin{lemma}\label{lem:linked_app}
An appendage node $\tau$ is in the $j^{th}$ linked appendage subnetwork $T_j$ if and only if there is a simple cycle $C$ that contains $\tau$ and a node in the simple subnetwork $\LL_j''$, but $C$ does not contain any super-simple node.
\end{lemma}

\begin{proof}
Recall that $\tau \in T_j$ if, for some complementary subnetwork $C_S$, the transitive  component of $\tau$ in $C_S$ is the union of $\tau$, nodes in $\LL'_j$, and non-super-appendage nodes.
The statement then follows because two nodes belong to the same transitive  component if and only if some cycle includes them.
\end{proof}

\subsection{Structural Subnetworks  of $\GG(\kappa)$}
\label{S:structural_subnetworks}

We start with a lemma that characterizes the super-simple nodes of $\GG$ that are super-simple nodes of $\GG(\kappa)$:

\begin{lemma} \label{lem:shared_super_simple}
Let  $\rho$ be a super-simple node of $\GG$ and let $\kappa$ be a node in $\GG$. Then $\rho$ is a super-simple node of $\GG(\kappa)$ if and only if  $\rho \preceq \sigma^u(\kappa)$. 
\end{lemma}

\begin{proof}
$(\Longleftarrow)$  
Suppose $\rho \preceq \sigma^u(\kappa)$.  Consider a simple path $p = \iota \pathto \kappa$.
Let $\sigma'$ be the last simple node on this path. We show that $\rho$ is super-simple in two parts.
If $\kappa$ is a simple node, then $\sigma^u(\kappa) = \kappa$ and $\sigma'=\kappa$ because it is the last simple node in $p$. Then $\rho \preceq \sigma'$ and $p$ contains $\rho$.   
If $\kappa$ is an appendage node, then  $\rho \preceq \sigma'$ since $\sigma^u(\kappa)$ is minimal. Consequently, $p$ contains $\rho$. Since the simple path from $\iota$ to $\kappa$ was arbitrary, $\rho$ is super-simple in $\GG(\kappa)$. 

$(\Longrightarrow)$
Suppose $\sigma^u(\kappa) \prec \rho$. Then there is a simple path $\iota \pathto \sigma^u(\kappa)$ which avoids $\rho$. If $\sigma^u(\kappa) \neq \kappa$, there is an appendage path $\sigma^u(\kappa) \pathto \kappa$, which by definition avoids $\rho$. There is a simple path $\iota \pathto \sigma^u(\kappa)$ which avoid $\rho$ since $\sigma^u(\kappa) \prec \rho$. The concatenation of these paths, $\iota \pathto \sigma^u(\kappa) \pathto \kappa$, avoids $\rho$ and shows $\rho$ is not super-simple in $\GG(\kappa)$. If $\sigma^u(\kappa) = \kappa$, then the simple path $\iota \pathto \sigma^u(\kappa)$ shows $\rho$ is not super-simple in $\GG(\kappa)$.
\end{proof}

The next result shows that a structural subnetwork of $\GG$ is a structural subnetwork of $\GG(\kappa)$ when its two super-simple nodes are super-simple nodes of $\GG(\kappa)$.

\begin{lemma} \label{lem:shared_structural}
Let $\kappa \in \GG$ be a node
and let $\LL_j$ be the $j^{th}$ structural subnetwork of $\GG$. Suppose $\rho_{j-1},\,\rho_j\in\LL_j$ are adjacent super-simple nodes of $\GG(\kappa)$. Then $\LL_j$ is a structural subnetwork of $\GG(\kappa)$. 
\end{lemma}

\begin{proof}
Let $\LL_j(\kappa)$ be the structural subnetwork of $\GG(\kappa)$ that has  super-simple nodes $\rho_{j-1}$ and $\rho_j$. Since $\rho_{j-1}$ and $\rho_j$ are adjacent super-simple nodes of both $\GG$ and $\GG(\kappa)$, the simple networks $\LL_j''$ and $\LL_j''(\kappa)$ are equal. We need to show that the linked appendage nodes $T_j$ and $T_j(\kappa)$ are equal. 
    
Suppose there exists $\tau \in T_j(\kappa)$ with $\tau \notin T_j$. By Lemma \ref{lem:linked_app}, there is a simple cycle $C$ that contains $\tau$, a node of $\LL_j'$, and no super-simple nodes of $\GG(\kappa)$. Since $\tau \notin T_j$, $C$ must contain a super-simple node $\rho$ of $\GG$ that is not super-simple in $\GG(\kappa)$. By Lemma \ref{lem:shared_super_simple}, this implies $\rho_j \prec \rho$.  But $C$ also contains a simple node $\sigma \in \LL_j'$, so this implies there is a simple path $\iota \pathto \sigma \pathto \tau \pathto \rho \pathto o$, contradicting $\tau$ being appendage in $\GG$. Therefore $\tau \in T_j$.
Reversing the roles of $T_j$ and $T_j(\kappa)$ in the above argument shows that if $\tau \in T_j$ then $\tau \in T_j(\kappa)$, completing the proof that $T_j = T_j(\kappa)$. 

Now $\rho_{j-1}$ and $\rho_j$ are super-simple in $\GG(\kappa)$, $\LL_j'' = \LL_j''(\kappa)$, and $T_j = T_j(\kappa)$. Therefore $\LL_j$ is  a structural subnetwork of $\GG(\kappa)$.
\end{proof}

\begin{remark} \label{R:struct_G_k}  \em
The network $\GG(\kappa)$ may have other structural subnetworks, but we study the homeostasis pattern induced by homeostasis subnetworks of $\GG$. Hence, we focus on whether a structural subnetwork of $\GG$ is still a structural subnetwork of $\GG(\kappa)$.

An appendage subnetwork $\AA$ of $\GG$ is not a structural subnetwork of $\GG(\kappa)$, because $\AA$ is transitive and thus has no super-simple nodes.
\END
\end{remark}

\subsection{Appendage Subnetworks of $\GG(\kappa)$}
\label{S:appendage_subnetworks}

\begin{lemma} \label{L:app_connect_order}
Let $\AA$ be an appendage subnetwork  of $\GG$
and let $\sigma_1, \sigma_2$ be simple nodes of $\GG$
with appendage paths $\AA \pathto \sigma_1$ and $\sigma_2 \pathto \AA$.  Then $\sigma_1 \preceq \sigma_2$. Moreover, either $\sigma_1 = \sigma_2$ is a super-simple node or $\sigma_1\prec \sigma_2$. 
\end{lemma}

\begin{proof}
 Since $\sigma_1,\sigma_2$ are simple nodes there are simple paths $\iota \pathto \sigma_2$ and 
 $\sigma_1 \pathto o$ that contain only simple nodes.  By contradiction, suppose that 
 $\sigma_2 \prec \sigma_1$. Then $\sigma_2$ is not on the path $\sigma_1\pathto o$ and 
 $\sigma_1$ is not on the path $\iota \pathto \sigma_2$. By concatenating paths we produce a simple path 
\[
  \iota \pathto \sigma_2 \pathto \AA \pathto  \sigma_1 \pathto o,
\]
which contradicts $\AA$ being an appendage subnetwork.

Next we show that it is impossible for both $\sigma_1$ and $\sigma_2$ to be contained in the 
same simple subnetwork $\LL''$.  
By contradiction, suppose there exists a simple subnetwork $\LL''$ such that $\sigma_1,\sigma_2 \in \LL''$.
There are three possibilities.

\begin{itemize}
\item[(a)] If there is a simple path $\iota \pathto \sigma_1 \pathto \sigma_2 \pathto o$, then there exists a simple path 
\[
    \iota \pathto \sigma_1 \pathto \AA \pathto \sigma_2 \pathto o,
\]
contradicting that $\AA$ is an appendage subnetwork. 

\item[(b)] If there is a simple path $\iota \pathto \sigma_2 \pathto \sigma_1 \pathto o$, then the cycle
\[
    \sigma_1 \pathto \AA \pathto \sigma_2 \pathto \sigma_1
\]
contradicts the fact that every node in $\AA$ is super-appendage. 

\item[(c)] If neither of the above paths exists, then there is a simple path 
$\iota \pathto \sigma_1 \pathto o$ that avoids $\sigma_2$ and a path 
$\iota \pathto \sigma_2 \pathto o$ that avoids $\sigma_1$. In this case, the path 
\[
    \iota \pathto \sigma_1 \pathto \AA \pathto \sigma_2 \pathto o
\]
contradicts the fact that $\AA$ is an appendage subnetwork.
\end{itemize}
We can now conclude that either $\sigma_1 = \sigma_2$ is a super-simple node or 
$\sigma_1\prec \sigma_2$.
\end{proof}

\begin{lemma}\label{L:A_induces}
Let $\AA$ be an appendage subnetwork of $\GG$
and let $\kappa \in \GG$ be a node that is not in $\AA$. If there is a path from $\AA$ to $\kappa$ and every such path passes through a super-simple node $\rho$ of $\GG$ satisfying $\sigma^d(\AA) \preceq \rho \preceq \sigma^u(\kappa)$, then $\AA$ is an appendage subnetwork of $\GG(\kappa)$.
\end{lemma}

\begin{proof}
First we show that the nodes of $\AA$ are appendage in $\GG(\kappa)$. 
By contradiction, suppose $\tau \in \AA$ is a simple node of $\GG(\kappa)$. 
Let $p = \iota \pathto \tau \pathto \kappa$ be a simple path in $\GG(\kappa)$. Let $\sigma_1$ be the last simple node of $\GG$ on the path $\iota \pathto \tau$. Let $\sigma_2$ and $\rho'$ be the first simple node and first super-simple node of $\GG$ on the path $\tau \pathto \kappa$, respectively. 

If $\rho'\preceq \sigma_1$ then the path $\iota \pathto \sigma_1$ passes through $\rho'$, contradicting that $p$ is a simple path.
Otherwise, suppose $\sigma_1 \prec \rho'$.
Since there are appendage paths $\sigma_1 \pathto \AA$ and $\AA\pathto \sigma_2$ and $\sigma_1 \neq \sigma_2$, by Lemma \ref{L:app_connect_order}, $\sigma_2 \prec \sigma_1$.  There is therefore a super-simple node $\rho''$ of $\GG$ with $\sigma_2 \preceq \rho'' \preceq \sigma_1$. The segments $\iota \pathto \sigma_1$ and $\sigma_2 \pathto \rho'$ of $p$ thus both contain $\rho''$, contradicting that $p$ is a simple path. We conclude that each node $\tau \in \AA$ is an appendage node of $\GG(\kappa)$. 

\medskip

Next we show that if there is a cycle $\AA \pathto \tau \pathto \AA$ consisting only of appendage nodes of $\GG(\kappa)$, then $\tau$ is an appendage node of $\GG$. 
Suppose not, then there exists a cycle $\AA \pathto \tau \pathto \AA$ contains a simple node of $\GG$ but only appendage nodes of $\GG(\kappa)$. Let $\sigma$ be the first simple node of $\GG$ on the cycle, starting from $\AA$. Since $\sigma$ is not a simple node of $\GG(\kappa)$, by Lemma \ref{lem:shared_super_simple}, either $\sigma\succ \sigma^u(\kappa)$ or $\sigma$ and $\sigma^u(\kappa)$ are incomparable. On the other hand, there is an appendage path from $\AA$ to $\sigma$ which implies $\sigma \preceq \sigma^d(\AA)$ or $\sigma$ and $\sigma^d(\AA)$ are incomparable. But this implies $\sigma^u(\kappa) \prec \sigma^d(\AA)$, which contradicts the assumption $\sigma^d(\AA) \preceq \rho \preceq \sigma^u(\kappa)$.

Now we claim that there is no appendage subnetwork $\AA'$ of $\GG(\kappa)$, such that $\AA \subsetneq \AA'$.
Suppose not, there exists such a appendage subnetworks $\AA'$. Then $\AA'$ is a a transitive component consisting of only appendage nodes of $\GG$.
Hence, $\AA'$ must be an appendage subnetwork of $\GG$ containing $\AA$. This contradicts $\AA$ an appendage subnetwork of $\GG$ that is not transitive with other appendage subnetworks.

\medskip

Finally, we show that $\AA$ is a transitive component of super-appendage nodes of $\GG(\kappa)$; that is, $\AA$ is an appendage subnetwork of $\GG(\kappa)$. 
Suppose not, then each node $\tau \in \AA$ is linked to a simple subnetwork $\LL''$ of $\GG(\kappa)$. By Lemma \ref{lem:linked_app}, there is a simple cycle $C$ 
that avoids super-simple node of $\GG(\kappa)$ and contains 
$\tau$ as well as a simple node of $\GG(\kappa)$.
Since $\AA$ is a transitive\  component of the appendage nodes of $\GG$, $C$ must contain a simple node of $\GG$. But $\tau$ is not a linked appendage node of $\GG$, so $C$ must contain a super-simple node $\rho'$ of $\GG$. Since $\rho'$ is not super-simple in $\GG(\kappa)$, by Lemma \ref{lem:shared_super_simple} $\rho\succ \sigma^u(\kappa)$. Let $\sigma$ be the first simple node of $\GG$ on the path $\AA \pathto \rho'$ in $C$. We have $\sigma \preceq \sigma^d(\AA)$ or $\sigma$ and $\sigma^d(\AA)$ are incomparable. But then the path $\sigma \pathto \rho'$ passes through $\rho$. Since $\rho \preceq \sigma^u(\kappa)$, $\rho$ is a super-simple node of $\GG(\kappa)$ by Lemma \ref{lem:shared_super_simple}, contradicting that $C$ avoids super-simple nodes of $\GG(\kappa)$. We conclude that $\AA$ is an appendage subnetwork of $\GG(\kappa)$.
\end{proof}

\begin{lemma}\label{L:avoiding_cycle}
Let $\AA$ be an appendage subnetwork of $\GG$ and let $\kappa$ be a node in $\GG$. Suppose a cycle 
$C$ in $\GG$ contains some node $\tau \in \AA$ and some node that is not in $\AA$.  Suppose also that $C$ 
avoids super-simple nodes of $\GG(\kappa)$. Then $\AA$ is not an appendage subnetwork of $\GG(\kappa)$. 
\end{lemma}

\begin{proof}
The cycle $C$ can take three possible forms, and we will show that $\AA$ is not an appendage subnetwork of $\GG(\kappa)$ in each case.

\begin{enumerate}[(a)]
\item {\bf $C$ consists of simple nodes of $\GG(\kappa)$}.
Since $C$ contains a node $\tau \in \AA$, this immediately implies $\AA$ is not an appendage subnetwork of $\GG(\kappa)$. 
 
\item {\bf $C$ consists of appendage nodes of $\GG(\kappa)$}.
A node $\tau \in \AA$ forms a cycle with appendage nodes of $\GG(\kappa)$ that are not in $\AA$. Consequently, $\AA$ is not a transitive component of appendage nodes of $\GG(\kappa)$ and therefore not an appendage subnetwork.
 
\item {\bf $C$ consists of both simple and appendage nodes of $\GG(\kappa)$}.
By Lemma \ref{lem:linked_app}, any appendage node $\eta$ of $\GG(\kappa)$ on $C$ is linked to the simple nodes of $\GG(\kappa)$ on $C$ and hence $\eta$ is a non-super-appendage node. Therefore any node $\tau \in \AA$ on $C$ is either simple or non-super-appendage in $\GG(\kappa)$. We conclude that $\AA$ is not an appendage subnetwork of $\GG(\kappa)$. \qedhere
\end{enumerate}
\end{proof}

\begin{lemma} \label{L:kappa_prec_AA}
Let $\AA$ be an appendage subnetwork of $\GG$ and let $\kappa$ be a node in $\GG$. If $\sigma^u(\kappa) \prec \sigma^d(\AA)$, then $\AA$ is not an appendage subnetwork of $\GG(\kappa)$.
\end{lemma}

\begin{proof}
By Lemma \ref{L:app_connect_order}, $\sigma^d(\AA) \preceq \sigma^u(\AA)$. 
Then there exists a simple cycle $C = \AA \pathto \sigma^d(\AA) \pathto \sigma^u(\AA) \pathto \AA$ where two segments $\AA \pathto \sigma^d(\AA)$ and $\sigma^u(\AA) \pathto \AA$ are appendage paths while every simple node of $\GG$ on $C$ is in the segment $\sigma^d(\AA) \pathto \sigma^u(\AA)$.

Now we claim that $C$ does not contain a super-simple node of $\GG(\kappa)$.
Given the claim, Lemma \ref{L:avoiding_cycle} implies $\AA$ is not an appendage subnetwork of $\kappa$. Thus it remains to prove the claim.

Consider any simple node $\sigma$ of $\GG$ on the segment $\sigma^d(\AA) \pathto \sigma^u(\AA)$ in $C$. 
Since $\sigma^u(\kappa) \prec \sigma^d(\AA)$, this implies that $\sigma^u(\kappa) \prec \sigma$. Thus there is a simple path $\iota \pathto \sigma^u(\kappa)$ that avoids every simple node of $\GG$ on $C$. 

If $\kappa$ is a simple node, then $\sigma^u(\kappa) = \kappa$. This gives an input-output path in $\GG(\kappa)$ that avoids every node in $C$, verifying the claim. 
If $\kappa$ is an appendage node, then $\sigma^u(\kappa) \neq \kappa$. There is an appendage path $\sigma^u(\kappa) \pathto \kappa$. For the sake of contradiction, suppose there is a node $\tau$ on $\sigma^u(\kappa) \pathto \kappa$ that is also on $C$. Either $\tau$ is on the segment $\AA \pathto \sigma^d(\AA)$ or the segment $\sigma^u(\AA) \pathto \AA$. If $\tau$ is on $\AA \pathto \sigma^d(\AA)$, then there is an appendage path $\sigma^u(\kappa) \pathto \tau \pathto \sigma^d(\AA)$. Since $\sigma^u(\kappa)\prec \sigma^d(\AA)$, this gives an input-output simple path 
\[
    \iota \pathto \sigma^u(\kappa) \pathto \tau \pathto \sigma^d(\AA) \pathto o
\]
which contradicts that $\tau$ is appendage in $\GG$. If $\tau$ is on $\sigma^u(\AA) \pathto \AA$, then this gives an appendage path $\sigma^u(\kappa) \pathto \tau \pathto \AA$. Since $\sigma^u(\kappa) \prec \sigma^u(\AA)$, this contradicts that $\sigma^u(\AA)$ is a minimal simple node with an appendage path to $\AA$. We conclude that the appendage path $\sigma^u(\kappa)\pathto \kappa$ does not contain a node of $C$. Therefore, the simple path $\iota \pathto \sigma^u(\kappa) \pathto \kappa$ is an input-output path in $\GG(\kappa)$ that avoids $C$. This proves the claim.  
\end{proof}

\begin{lemma} \label{L:kappa_equal_AA}
Let $\AA$ be an appendage subnetwork   of $\GG$ and let $\kappa$ be a node in $\GG$.
  If both $\sigma^u(\kappa)$ and $\sigma^d(\AA)$ are nodes in a simple subnetwork $\LL_j''$ of $\GG$, 
  then $\AA$ is not an appendage subnetwork of $\GG(\kappa)$. 
\end{lemma}

\begin{proof}
The proof proceeds by stating two claims, proving the lemma from the claims, and then proving the claims.

\medskip

\noindent
\textbf{Claim 1:} There exists a simple path $p = \rho_{j-1} \pathto \rho_j$, such that $\rho_{j-1}$ is the only super-simple node of $\GG(\kappa)$ on $p$.

\medskip

\noindent
\textbf{Claim 2:} The appendage path $\sigma^u(\AA) \pathto \AA$ does not contain super-simple nodes of $\GG(\kappa)$.

\medskip

Assuming the claims, we consider the relation between $\sigma^d(\AA)$ and the path $p = \rho_{j-1} \pathto \rho_j$ in the first claim and split the proof into two cases as follows.

\paragraph{{\bf Assume $\sigma^d(\AA)$ is on the path $p$}.}
Then there is a simple path $\sigma^d(\AA) \pathto \rho_j$ avoids super-simple nodes of $\GG(\kappa)$. By Lemma \ref{L:app_connect_order}, $\sigma^d(\AA) \prec \rho_j \preceq \sigma^u(\AA)$ so there is a simple path $\sigma^d(\AA) \pathto \sigma^u(\AA)$ which avoids super-simple nodes of $\GG(\kappa)$. If an appendage path $\AA \pathto \sigma^d(\AA)$ contains a super-simple node of $\GG(\kappa)$, then the simple path $\iota \pathto \sigma^d(\AA) \pathto \AA \pathto \kappa$ which contains this super-simple node shows $\AA$ is not an appendage subnetwork of $\GG(\kappa)$. If an appendage path $\AA \pathto \sigma^d(\AA)$ does not contain a super-simple node of $\GG(\kappa)$ then we have found a super-simple node avoiding cycle 
\[
    C = \AA \pathto \sigma^d(\AA) \pathto \sigma^u(\AA) \pathto \AA,
\]
and by Lemma \ref{L:avoiding_cycle}, $\AA$ is not an appendage subnetwork of $\GG(\kappa)$. 

\paragraph{{\bf Assume every path $\rho_{j-1} \pathto \sigma^d(\AA) \pathto \rho_j$ contains a super-simple node of $\GG(\kappa)$ other than $\rho_{j-1}$}.}
If there is a path $\sigma^d(\AA) \pathto \rho_j$ that does not contain a super-simple node of $\GG(\kappa)$, then there is a simple path $\sigma^d(\AA) \pathto \sigma^u(\AA)$ that avoids super-simple nodes of $\GG(\kappa)$. This is the same situation deduced in case (1), so $\AA$ is not an appendage subnetwork of $\GG(\kappa)$. Suppose there is a path $\sigma^d(\AA) \pathto \rho_j$ which contains a super-simple node $\rho(\kappa)$ of $\GG(\kappa)$. Since there is a path $p = \rho_{j-1} \pathto \rho_j$ that avoids super-simple nodes of $\GG(\kappa)$, there is a simple path 
\[
    p_1 = \iota \pathto \rho_{j-1} \pathto \sigma^u(\AA) \pathto \AA \pathto \sigma^d(\AA) \pathto \rho(\kappa).
\]
Let $p_2 = \rho(\kappa) \pathto \kappa$ be a simple path. If $p_1$ and $p_2$ do not overlap, then the concatenation of $p_1$ and $p_2$ shows $\AA$ is not an appendage subnetwork of $\GG(\kappa)$. If $p_1$ and $p_2$ overlap, they overlap in an appendage node $\tau$ of $\GG$. If $\tau$ is on the segment $\sigma^u(\AA) \pathto \AA$, then this implies there is an appendage path from $\til{\LL}_j$ to $\AA$, contradicting that $\sigma^u(\AA)$ is minimal. Therefore, $\tau$ is on the segment $\AA \pathto \sigma^d(\AA)$. Let $\tau$ be the first node on $p_1$ which is on $p_2$. The simple path 
\[
    p_1 = \iota \pathto \rho_{j-1} \pathto \sigma^u(\AA) \pathto \AA \pathto \tau \pathto \kappa
\]
shows that $\AA$ is not an appendage subnetwork of $\GG(\kappa)$.

\medskip

We now proceed to prove the claims. 

\medskip

\noindent
\textbf{Proof of Claim 1:} Suppose $\rho_{j-1} \preceq \sigma^u(\kappa)$ is the most downstream super-simple node of $\GG$.
Consider a path $\rho_{j-1} \pathto \sigma^u(\kappa)$.
From Lemma \ref{lem:shared_super_simple}, $\rho_{j-1}$ is a super-simple node of $\GG(\kappa)$.
If there is no other super-simple node of $\GG(\kappa)$ on this path, then we take $p$ to be the path $\rho_{j-1} \pathto \sigma^u(\kappa) \pathto \rho_j$.
If there exists other super-simple nodes of $\GG(\kappa)$ on $\rho_{j-1} \pathto \sigma^u(\kappa)$, let $\sigma$ be the first super-simple node of $\GG(\kappa)$ on this path. Since $\sigma$ is a simple node of $\GG$, there is a path $p = \rho_{j-1} \pathto \rho_j$ which avoids $\sigma$. If $p$ contained any super-simple node of $\GG(\kappa)$ other than $\rho_{j-1}$, then this would contradict that $\sigma$ is super-simple in $\GG(\kappa)$. Therefore $p$ is the desired path. 

\medskip

\noindent
\textbf{Proof of Claim 2:} If $\kappa$ is a simple node so that $\sigma^u(\kappa) = \kappa$, then $\kappa \prec \sigma^u(\AA)$ implies there is a simple path $\iota \pathto \kappa$ which avoids $\sigma^u(\AA)$, validating the claim. 
If $\kappa$ is an appendage node, suppose $\sigma^u(\AA) \pathto \AA$ contains a super-simple node $\rho(\kappa)$ of $\GG(\kappa)$. Then the appendage path $\sigma^u(\kappa) \pathto \kappa$ passes through $\rho(\kappa)$ and in particular there is an appendage path $\sigma^u(\kappa) \pathto \rho(\kappa) \pathto \AA$, which contradicts that $\sigma^u(\AA)$ is a minimal simple node with an appendage path to $\AA$.  
\end{proof}

\begin{remark} \label{R:app_G_k} \rm
 As in Remark \ref{R:struct_G_k}, 
$\GG(\kappa)$ may have other appendage subnetworks, but we only need to check whether an appendage subnetwork of $\GG$ is still an appendage subnetwork of $\GG(\kappa)$.
Any structural subnetwork $\LL_j$ of $\GG$ is not an appendage subnetwork of $\GG(\kappa)$, since $\LL_j$ is not transitive, that is, there is no path from $\rho_j$ to $\rho_{j-1}$ in $\LL_j$.
\END
\end{remark}

\section{Structural Homeostasis Pattern}
\label{sec:struct_pattern}

Here we prove Theorems \ref{thm:struct_to_struct} and \ref{thm:struct_to_app}, which characterize the structural homeostasis pattern. We separate the proof of each theorem into two propositions, one that identifies the nodes which are induced by a structural subnetwork and the other that identifies the nodes that are not induced by a structural subnetwork.
We recall that we denote the nodes of $\PP_\sS$ by $\VV_\sS$,  
 the backbone nodes or the appendage components of $\PP$ by $\til{\VV}$,
 and the corresponding homeostasis subnetwork of $\GG$ by $\VV$.

The following two propositions identify 
nodes in structural and appendage subnetworks
that are not induced by $\LL_j$.

\begin{proposition}\label{prop:struct_not_induce_struct}
Let $\LL_j$ be the $j^{th}$ structural subnetwork of $\GG$, and let $\til{\VV} \in \PP_{\sS}$ be a backbone node of $\PP_\sS$ that is upstream from $\rho_j$. 
If $\kappa \in \VV$ and $\kappa \neq \rho_j$, then $\LL_j\not\Rightarrow \kappa$.     
\end{proposition}

\begin{proof}
Let $\kappa \in \VV$ be a simple node of $\GG$ and $\sigma^u(\kappa)$ be as in Definition~\ref{D:sigma(kappa)}.
Since $\til{\VV}$ is upstream from $\rho_j$ in $\PP_{\sS}$
and $\kappa \neq \rho_j$, $\sigma^u(\kappa) \prec \rho_j$. By Lemma \ref{lem:shared_super_simple}, $\rho_j$ is not a super-simple node of $\GG(\kappa)$. 
Therefore, $\LL_j$ is not a structural subnetwork of $\GG(\kappa)$, so $\LL_j \not\Rightarrow \kappa$. 
\end{proof}

\begin{proposition}\label{prop:struct_not_induce_app}
Let $\LL_j$ be the $j^{th}$ structural subnetwork of $\GG$ and $\AA$ be an appendage subnetwork of $\GG$. Let 
$\VV_\sS \to \til{\AA}$ 
be the arrow in $\PP$ from 
$\PP_\sS$ to $\til{\AA}$ and suppose 
 $\VV_\sS$ is strictly upstream of $\rho_j$. If $\tau \in \AA$, then $\LL_j \not\Rightarrow \tau$.
\end{proposition}

\begin{proof} 
Let $\tau \in \AA$ be an appendage node of $\GG$ and $\sigma^u(\tau)$ be as in Definition~\ref{D:sigma(kappa)}. 
We have $\sigma^u(\tau) \in \VV_\sS$ and $\VV_\sS$ is strictly upstream of $\rho_j$, this implies 
$\sigma^u(\tau) \prec \rho_j$.
By Lemma \ref{lem:shared_super_simple}, $\rho_j$ is not a super-simple node of $\GG(\tau)$. Therefore, $\LL_j$ is not a structural subnetwork of $\GG(\tau)$ and so $\LL_j \not\Rightarrow \tau$. 
\end{proof}

The next two propositions identify 
nodes in structural and appendage subnetworks
that are induced by $\LL_j$, respectively. 

\begin{proposition}\label{prop:struct_induce_struct}
Let $\LL_j$ be the $j^{th}$ structural subnetwork of $\GG$ and $\til{\VV}$ be a 
backbone node of $\PP_\sS$ that is strictly
downstream of $\til{\LL}_j$. 
If $\kappa \in \VV$, then $\LL_j$ induces $\kappa$, which is denoted by $\LL_j \Rightarrow \kappa$. 
\end{proposition}

\begin{proof}
Since $\til{\VV}$ is strictly downstream of $\til{\LL}_j$ 
and $\kappa \in \VV$ be a simple node of $\GG$, then we have
$\rho_j \preceq \sigma^u(\kappa)$. By Lemma \ref{lem:shared_super_simple}, $\rho_{j-1}$ and $\rho_j$ are adjacent super-simple nodes of $\GG(\kappa)$. By Lemma \ref{lem:shared_structural}, $\LL_j$ is a structural subnetwork of $\GG(\kappa)$ and so $\LL_j \Rightarrow \kappa$. 
\end{proof}

\begin{proposition}\label{prop:struct_induce_app}
Let $\LL_j$ be the $j^{th}$ structural subnetwork of $\GG$ and $\AA$ be an appendage subnetwork of 
$\GG$. Let 
$\VV_\sS \to \til{\AA}$ be the arrow in $\PP$ from 
$\PP_\sS$ to $\til{\AA}$ and suppose 
$\VV_\sS$ is strictly
downstream from $\til{\LL}_j$. If $\tau \in \AA$, then $\LL_j \Rightarrow \tau$.
\end{proposition}

\begin{proof}
Let $\tau \in \AA$ be an appendage node of $\GG$ and $\sigma^u(\tau)$ be as in Definition~\ref{D:sigma(kappa)}. 
Then $\sigma^u(\tau) \in \VV_\sS$ and $\VV_\sS$ is strictly downstream from  $\til{\LL}_j$, this implies 
$\rho_j \preceq \sigma^u(\tau)$.
By Lemma \ref{lem:shared_super_simple}, $\rho_{j-1}$ and $\rho_j$ are adjacent super-simple nodes of $\GG(\kappa)$. 
Lemma \ref{lem:shared_structural} then implies that $\LL_j$ is a structural subnetwork of $\GG(\tau)$ and so $\LL_j \Rightarrow \tau$. 
\end{proof}

\begin{proof}[Proof of Theorems
\ref{thm:struct_to_struct} and \ref{thm:struct_to_app}]
Recall from Definition \ref{D:homeo_inducing_pattern}, suppose $\rho\in\PP$ is a super-simple node and
$\til{\VV}_1, \til{\VV}_2 \in \PP$ are non-super-simple nodes with $\VV_1, \VV_2 \subset \GG$ being their corresponding homeostasis subnetworks. We say  $\til{\VV}_1 \Rightarrow \til{\VV}_2$ if and only if $\VV_1 \Rightarrow \VV_2$, and $\til{\VV}_1 \Rightarrow \rho$ if and only if $\VV_1 \Rightarrow \rho$.

Then we conclude Theorem \ref{thm:struct_to_struct} from Propositions \ref{prop:struct_not_induce_struct} 
and \ref{prop:struct_induce_struct}. Theorem \ref{thm:struct_to_app} follows from Propositions 
\ref{prop:struct_not_induce_app} and \ref{prop:struct_induce_app}. 
\end{proof}

\section{Appendage Homeostasis Pattern}
\label{sec:app_pattern}

Here we prove Theorems \ref{thm:app_to_struct} and \ref{thm:app_to_app}. Each theorem will follow from a series of propositions. The propositions are organized according to whether they make a statement about which nodes an appendage subnetwork $\AA$ induces or about which nodes $\AA$ does not induce. 

The following two propositions identify 
nodes in structural and appendage subnetworks
that are induced by $\AA$.

\begin{proposition}\label{prop:app_to_struct}
Let $\AA$ be an appendage subnetwork of $\GG$ and $\LL_j$ be the $j^{th}$ structural subnetwork of $\GG$. Let 
$\til{\AA} \to \VV_\sS$ be the arrow in $\PP$ from $\til{\AA}$ to $\PP_\sS$ and suppose $\VV_\sS$ is strictly
upstream of $\til{\LL}_j$. If $\kappa \in \LL_j$, then $\AA$ induces $\kappa$, which is 
denoted by $\AA \Rightarrow \kappa$. 
\end{proposition}

\begin{proof}
Since $\til{\AA} \to \VV_\sS \in \PP$ and $\VV_\sS$ is strictly
upstream of $\til{\LL}_j$, 
 then $\sigma^d(\AA) \in \VV_\sS$ and $\sigma^d(\AA) \preceq \rho_{j-1}$.
On the other hand, $\kappa \in \LL_j$ implies $\rho_{j-1} \preceq \sigma^u(\kappa)$. Applying Lemma \ref{L:A_induces}, we conclude $\AA \Rightarrow \kappa$.
\end{proof}

\begin{proposition} \label{prop:app_to_app}
Let $\AA_1$ and $\AA_2$ be distinct appendages subnetworks  of $\GG$.
Let $\til{\AA}_1 \to \VV_{\sS_1}$ be the arrow in $\PP$ from $\til{\AA}_1$ to $\PP_\sS$
and $\VV_{\sS_2} \to \til{\AA}_2$ be the arrow in $\PP$ from $\PP_\sS$ to $\til{\AA}_2$.
Suppose there is a path from $\til{\AA}_1$ to $\til{\AA}_2$ in $\PP$ and every such path passes through a super-simple node $\rho$ satisfying $\VV_{\sS_1} \preceq \rho \preceq \VV_{\sS_2}$.
If $\tau \in \AA_2$, then $\AA_1 \Rightarrow \tau$. 
\end{proposition}

\begin{proof}
Since $\sigma^d(\AA_1)$ is the maximal (downstream) simple node with an appendage path from $\AA_1$, we have $\sigma^d(\AA_1) \in \VV_{\sS_1}$.
Let $\tau \in \AA_2$ be an appendage node of $\GG$ and $\sigma^u(\tau)$ be as in Definition~\ref{D:sigma(kappa)}. 
Then $\sigma^u(\tau) \in \VV_{\sS_2}$ and every path from $\til{\AA}_1$ to $\til{\AA}_2$ in $\PP$ passes through a super-simple node $\rho$ satisfying $\sigma^d(\AA_1)\preceq\rho\preceq \sigma^u(\kappa)$.
Therefore, it follows from Lemma \ref{L:A_induces} that 
$\AA \Rightarrow \tau$.
\end{proof}

The next five propositions identify 
nodes in structural and appendage subnetworks 
that are not induced by $\AA$. 

\begin{proposition}\label{prop:app_not_struct}
Let $\AA$ be an appendage subnetwork of $\GG$ and let $\LL_j'$ be the $j^{th}$ augmented simple subnetwork of $\GG$. Let $\til{\AA} \to \VV_\sS$ be the arrow in $\PP$ from $\til{\AA}$ to $\PP_\sS$
and suppose 
$\VV_\sS$
is downstream from or equal to $\til{\LL}_j$. If $\kappa \in \LL_j'$, then $\AA \not \Rightarrow \kappa$. 
\end{proposition}

\begin{proof}
Let  $\kappa$ be a node in $\LL_j'$. If $\VV_\sS$ is strictly
downstream from $\til{\LL}_j$, then $\sigma^u(\kappa) \prec \sigma^d(\AA)$ and by Lemma \ref{L:kappa_prec_AA}, $\AA \not\Rightarrow \kappa$. If $\til{\VV}=\til{\LL}_j$, then $\sigma^d(\AA)\in \LL_j'$ and $\sigma^u(\kappa) \in \LL_j'$ so by Lemma \ref{L:kappa_equal_AA} $\AA \not\Rightarrow \kappa$. 
\end{proof}

\begin{proposition}\label{prop:app_not_super_simple}
Let $\AA$ be an appendage subnetwork of $\GG$ and let $\rho_j$ be the $j^{th}$ super-simple node of  $\GG$. Let 
$\til{\AA} \to \VV_\sS$ be the arrow in $\PP$ from $\til{\AA}$ to $\PP_\sS$
and suppose $\VV_\sS$ is strictly
downstream from $\rho_j$. 
Then $\AA \not\Rightarrow \rho_j$. 
\end{proposition}

\begin{proof}
If $\VV_\sS$ is strictly
downstream from $\rho_j$ then $\rho_j \prec \sigma^d(\AA)$. By Lemma \ref{L:kappa_prec_AA}, $\AA \not\Rightarrow \rho_j$. 
\end{proof}

\begin{proposition} \label{prop:app_not_itself}
Let $\AA$ be an appendage subnetwork of $\GG$ and $\tau \in \AA$.
Then $\AA \not\Rightarrow \tau$. 
\end{proposition}

\begin{proof}
Since the node $\tau$ is the output node of $\GG(\tau)$, thus it is a simple node of $\GG(\tau)$. We conclude that $\AA$ is not an appendage subnetwork of $\GG(\tau)$, so $\AA \not\Rightarrow \tau$. 
\end{proof}

\begin{proposition}\label{prop:app_no_path}
Let $\AA_1$ and $\AA_2$ be distinct
appendage subnetworks of $\GG$ and $\tau \in \AA_2$.
Let $\til{\AA}_1 \to \VV_{\sS_1}$ be the arrow in $\PP$ from $\til{\AA}_1$ to $\PP_\sS$
and $\VV_{\sS_2} \to \til{\AA}_2$ be the arrow in $\PP$ from $\PP_\sS$ to $\til{\AA}_2$.
If there is no path from $\til{\AA}_1$ to $\til{\AA}_2$ in the pattern network $\PP$ or $\VV_{\sS_2} \prec \VV_{\sS_1}$, then $\AA_1 \not\Rightarrow \tau$. 
\end{proposition}

\begin{proof} 
Let $\tau \in \AA_2$ be an appendage node of $\GG$ and $\sigma^u(\tau)$ be as in Definition~\ref{D:sigma(kappa)}.
We consider $\sigma^d(\AA_1)$, which is the maximal (downstream) simple node with an appendage path from $\AA_1$.
Since there is no path from $\til{\AA}_1$ to $\til{\AA}_2$ in $\PP$ or 
$\VV_{\sS_2} \prec \VV_{\sS_1}$, we have $\sigma^u(\tau_2) \prec \sigma^d(\AA_1)$.
By Lemma \ref{L:kappa_prec_AA}, $\AA \not\Rightarrow \tau_2$. 
\end{proof}

\begin{proposition} \label{prop:app_bad_path}
Let $\AA_1$ and $\AA_2$ be distinct appendage subnetworks of $\GG$ and $\tau \in \AA_2$. 
If there is a path from $\til{\AA}_1$ to $\til{\AA}_2$ in $\PP$ that does not contain a super-simple node, then $\AA_1 \not\Rightarrow \tau$. 
\end{proposition}

\begin{proof}
There are two ways for a path $p \in \PP$ from $\til{\AA}_1$ to $\til{\AA}_2$ to avoid a super-simple node, where $p$ contains backbone nodes or has no backbone node.
We consider each case separately.

\paragraph{The path $p$ contains a backbone node $\til{\LL}_j$.} 
Then $p$ is the concatenation of two appendage paths: $\til{\AA}_1 \pathto \til{\LL}_j$ 
and $\til{\LL}_j \pathto \til{\AA}_2$. Let $\tau \in \AA_2$ be appendage node.
Suppose $\sigma^d(\AA_1) \prec \sigma^u(\tau)$, there exists a super-simple node $\rho$ of $\GG$ such that $\sigma^d(\AA_1) \preceq \rho \preceq \sigma^u(\tau)$. Hence the path $p$ from $\til{\AA}_1$ to $\til{\AA}_2$ must contain a super-simple node $\rho$, which contradicts with the assumption.
Thus we derive $\sigma^u(\tau) \preceq \sigma^d(\AA_1)$.
Therefore, either $\sigma^u(\tau) \prec \sigma^d(\AA_1)$ and Lemma \ref{L:kappa_prec_AA} implies $\AA_1 \not\Rightarrow \tau_2$, or $\sigma^u(\tau)$ and $\sigma^d(\AA_1)$ belong to $\LL_j''$ and by Lemma \ref{L:kappa_equal_AA} $\AA_1 \not\Rightarrow \tau_2$.

\paragraph{The path $p$ contains no backbone node.} 
Then there is an appendage path $\AA_1 \pathto \tau_2$. Since $\tau_2$ is an appendage node, any simple path $p_1:=\iota \pathto \sigma^u(\AA_1)$ avoids $\tau_2$. There is an appendage path $p':=\sigma^u(\AA_1) \pathto \AA_1$. Since both $p$ and $p'$ contain a node of $\AA_1$ and the nodes of $\AA_1$ are path-connected, we may assume $p$ and $p'$ share a node $\tau$. \textit{A priori}, $\tau$ need not be a node of $\AA_1$, but, since $p$ consists only of appendage nodes, $\tau$ is an appendage node. Furthermore, the sub-path $\tau\pathto \AA_1$ on $p'$ and the sub-path $\AA_1 \pathto \tau$ on $p$ implies $\tau$ belongs to the transitive component of $\AA_1$ and thus $\tau \in \AA_1$.  Then there is a simple path $\iota \pathto \sigma^u(\AA_1) \pathto \tau \pathto \tau_2$ so that $\tau$ is a simple node of $\GG(\tau_2)$. Since $\tau \in \AA_1$ is not appendage in $\GG(\tau_2)$, we conclude that $\AA_1$ is not an appendage subnetwork of $\GG(\tau_2)$ and so $\AA_1 \not\Rightarrow \tau_2$. 
\end{proof}

\begin{proof}[Proof of Theorems \ref{thm:app_to_struct} and \ref{thm:app_to_app}]
Recall the relation of homeostatic induction between $\GG$ and $\PP$ in Definition \ref{D:homeo_inducing_pattern}, we conclude Theorem \ref{thm:app_to_struct} from Propositions \ref{prop:app_to_struct}, \ref{prop:app_not_struct}, and \ref{prop:app_not_super_simple}. Similarly, Theorem \ref{thm:app_to_app} follows from Propositions \ref{prop:app_to_app}, \ref{prop:app_not_itself}, \ref{prop:app_no_path}, and \ref{prop:app_bad_path}.
\end{proof}

\section{Properties of the Induction Relation ($\Rightarrow$)}
\label{S:properties_induction}

In this section, we give three general results about the induction relation.
First, in Theorem~\ref{thm:pattern_from_subnetworks_intro} we prove that the induction relation is characterized by its behavior on homeostasis subnetworks.
Next, in Theorems~\ref{thm:always_related_intro} we prove that induction applies in at least one direction for distinct homeostasis subnetworks and that no subnetwork induces itself.
Finally, in  Theorem~\ref{thm:bijective_corresp_intro} we prove that distinct subnetworks have distinct homeostasis patterns. That is, the set of nodes induced by a homeostasis subnetwork is unique among all homeostasis subnetworks.

\begin{theorem} \label{thm:pattern_from_subnetworks_intro}
Suppose $\KK_1$ and $\KK_2$ are distinct homeostasis subnetworks of $\GG$.  Let $\kappa$ be a node of 
$\KK_2$ where $\kappa\neq o$.  If $\KK_1\Rightarrow \kappa$, then $\KK_1 \Rightarrow \KK_2$.
\end{theorem}

\begin{proof}
There are four possibilities for $\KK_1$ to induce $\kappa$.  Each of these is determined by 
the classification of $\KK_1$ as appendage or structural subnetwork and $\kappa$ as appendage 
or simple node.  The four possibilities are discussed next.

\paragraph{$\KK_1$ structural and $\kappa$ simple.} 
Then $\kappa \in \KK_2 = \LL$ and the proof follows from Propositions \ref{prop:struct_induce_struct} and \ref{prop:struct_not_induce_struct}. 
We remark that when $\kappa$ is a super-simple node, there are two structural subnetworks (two $\KK_2$'s) containing $\kappa$. Further, if $\kappa \in \KK_1$ then $\KK_1$ only induces the structural subnetworks which is distinct from $\KK_1$.
 
\paragraph{$\KK_1$ structural and $\kappa$ appendage.} 
Then $\kappa \in \KK_2 = \AA$ and the proof follows from Propositions \ref{prop:struct_induce_app} and \ref{prop:struct_not_induce_app}. 

\paragraph{$\KK_1$ appendage and $\kappa$ simple.}
Then $\kappa \in \KK_2 = \LL$ and the proof follows from Propositions \ref{prop:app_to_struct}, \ref{prop:app_not_struct}, and \ref{prop:app_not_super_simple}.

\paragraph{$\KK_1$ appendage and $\kappa$ appendage.} 
Then $\kappa \in \KK_2 = \AA$ and the proof follows from Propositions \ref{prop:app_to_app}, \ref{prop:app_not_itself}, \ref{prop:app_no_path}, and \ref{prop:app_bad_path}. 
\end{proof}

\begin{theorem} \label{thm:always_related_intro}
Let $\KK_1$ be a homeostasis subnetwork of $\GG$. Then generically $\KK_1 \not\Rightarrow \KK_1$.
Moreover, let $\KK_2$ be some other homeostasis subnetwork of $\GG$.
Then one of the following relations holds:
\begin{enumerate}[(a)]
\item $\KK_1 \Rightarrow \KK_2$ and $\KK_2 \not\Rightarrow \KK_1$,

\item $\KK_2 \Rightarrow \KK_1$ and $\KK_1 \not\Rightarrow \KK_2$,

\item $\KK_1 \Rightarrow \KK_2$ and $\KK_2 \Rightarrow \KK_1$.
\end{enumerate}
\end{theorem}

\begin{proof}

First we show that homeostasis subnetwork does not induce itself. Let $\KK_1$ be a homeostasis subnetwork.
If $\KK_1$ is a structural subnetwork, then Proposition \ref{prop:struct_not_induce_struct} implies $\KK_1 \not\Rightarrow \KK_1$.  If $\KK_1$ is an appendage subnetwork, then Proposition \ref{prop:app_not_itself} implies $\KK_1 \not\Rightarrow \KK_1$. 

\medskip

Next suppose $\KK_2$ is a homeostasis subnetwork of $\GG$ where $\KK_2 \neq \KK_1$.
We split the proof into three cases based on whether $\KK_1$ and $\KK_2$ are structural or appendage subnetworks. 

\paragraph{Both $\KK_1$ and $\KK_2$ are structural subnetworks.}  Without loss of generality we assume $\KK_1$ is strictly upstream from $\KK_2$. Theorem \ref{thm:struct_to_struct} implies that $\KK_1 \Rightarrow \KK_2$ and $\KK_2 \not\Rightarrow \KK_1$. 

\paragraph{Both $\KK_1$ and $\KK_2$ are appendage subnetworks.} 
Let $\VV^1_{max}, \ \VV^1_{min}$ be nodes in $\PP_\sS$ with arrows from and to $\KK_1$ in $\PP$, and let $\VV^2_{max}, \ \VV^2_{min}$ be nodes in $\PP_\sS$ with arrows from and to $\KK_2$ in $\PP$. These connections with $\KK_1, \KK_2$ in $\PP$ are
\[
\VV^1_{min} \to \til{\KK}_1 \quad \til{\KK}_1 \to \VV^1_{max} \qquad \VV^2_{min} \to \til{\KK}_2 \quad \til{\KK}_2 \to \VV^2_{max}
\]

First assume without loss of generality that there is an appendage path from $\KK_2$ to $\KK_1$. Theorem \ref{thm:app_to_app} implies $\KK_2 \not \Rightarrow \KK_1$.
Since $\KK_1$ and $\KK_2$ are transitive components of appendage nodes, so the existence of an appendage path $\KK_2 \pathto \KK_1$ precludes the existence of an appendage path $\KK_1 \pathto \KK_2$, and thus 
$\VV^1_{min} \preceq \VV^2_{min}$.
Then Lemma \ref{L:app_connect_order} shows either $\VV^1_{max} \prec \VV^1_{min}$ or $\VV^1_{max} = \VV^1_{min}$ is a super-simple node.  
Therefore every path from $\til{\KK}_1$ to $\til{\KK}_2$ follows
$\til{\KK}_1 \pathto \VV^1_{max} \pathto \VV^2_{min} \pathto \til{\KK}_2$, and it passes through a super-simple node $\rho$ with $\VV^1_{max} \preceq \rho \preceq \VV^2_{min}$. 
Theorem \ref{thm:app_to_app} shows $\KK_1 \Rightarrow \KK_2$.

\medskip

Second assume there is no appendage path between $\KK_1$ and $\KK_2$. With loss of generality, we further assume $\VV^1_{min} \preceq \VV^2_{min}$.
Similarly, Lemma \ref{L:app_connect_order} shows $\VV^1_{min} \prec \VV^1_{min}$ or $\VV^1_{min} = \VV^1_{min}$ is a super-simple node.  
Then every path from $\til{\KK}_1$ to $\til{\KK}_2$ passes through a super-simple node $\rho$ with $\VV^1_{max} \preceq \rho \preceq \VV^2_{min}$. 
Theorem \ref{thm:app_to_app} shows $\KK_1 \Rightarrow \KK_2$.
Moreover, if $\VV^1_{min} = \VV^2_{min} = \VV^2_{max}$ is a super-simple node, we also have $\KK_2 \Rightarrow \KK_1$.
    
\paragraph{$\KK_1$ is appendage subnetwork and $\KK_2$ is structural subnetwork.} Let $\VV^1_{max}, \ \VV^1_{min}$ be nodes in $\PP_\sS$ with arrows from and to $\KK_1$ in $\PP$. These connections with $\KK_1$ in $\PP$ are:
\[
\VV^1_{min} \to \AA \qquad \AA \to \VV^1_{max}
\]
Lemma \ref{L:app_connect_order} shows that either $\VV^1_{max} \prec \VV^1_{min}$ or $\VV^1_{max} = \VV^1_{min}$ is a super-simple node.  
Therefore either $\til{\KK}_2$ is upstream of $\VV^1_{min}$ or $\VV^1_{max}$ is upstream from $\til{\LL}_j$ in $\PP_{\sS}$. 
By Theorems \ref{thm:struct_to_app} and \ref{thm:app_to_struct}, either $\AA \Rightarrow \LL_j$ or $\LL_j \Rightarrow \AA$. 
We remark that a similar argument can be achieved when $\KK_2$ is appendage subnetwork and $\KK_1$ is structural subnetwork because of symmetry.
\end{proof}

Theorems~\ref{thm:pattern_from_subnetworks_intro} and \ref{thm:always_related_intro} imply
that each homeostasis subnetwork has a unique homeostasis pattern associated to it.  Specifically:

\begin{theorem} \label{thm:bijective_corresp_intro}
Let $\KK_1$ and $\KK_2$ be two distinct homeostasis subnetworks of $\GG$. 
Then the set of subnetworks induced by $\KK_1$ and the set of subnetworks induced by $\KK_2$ are distinct.
\end{theorem}

\begin{proof}
Consider any two distinct homeostasis subnetworks 
$\KK_1$ and $\KK_2$ of $\GG$, it is sufficient to find a homeostasis subnetwork $\KK \subset \GG$ such that 
$\KK_1 \Rightarrow \KK$ and $\KK_2 \not\Rightarrow \KK$. By Theorem \ref{thm:always_related_intro}, we have at least one of $\KK_1 \Rightarrow \KK_2$ and $\KK_2 \Rightarrow \KK_1$ holds.
We may assume without loss of generality that $\KK_1 \Rightarrow \KK_2$. 
Note that Theorem \ref{thm:always_related_intro} also shows that $\KK_2 \not \Rightarrow \KK_2$.
Therefore we set $\KK =\KK_2$ and thus $\KK_1 \Rightarrow \KK$ and $\KK_2 \not\Rightarrow \KK$.
\end{proof}

\section{Discussion}
\label{S:discussion}

Wang~\etal~\cite{WHAG21} show that an input-output network $\GG$ with an input parameter $\II$ can, 
under certain circumstances, lead to several different infinitesimal homeostasis types.  
Sections~\ref{sec:comb_charact}-\ref{S:properties_induction} show that each homeostasis type corresponds to a unique infinitesimal homeostasis pattern;  that is, a subset of nodes in $\GG$ vary homeostatically as $\II$ varies.

There are, are least, two relevant ways to modify the theory of homeostasis in input-output networks, which in turn open up new avenues for applications. 

The first modification concerns homeostasis in Gene Regulatory Networks (GRN).  
See Antoneli~\etal~\cite{AGS18, GRN} and Golubitsky and Stewart~\cite{GS23}.  
The important difference between a GRN and an input-output network is the assumption that each node (a `gene') in a GRN consists of a pair of state variables (the protein and mRNA concentrations).  
One way to deal with this `discrepancy' is o consider a subclass of input-output networks that we call PRN (Protein-mRNA Networks).
This notion leads to somewhat different infinitesimal homeostasis types and patterns. The general theory of homeostasis types and patterns in PRN is developed in \cite{GRN}.

A second modification considers input-output networks where the input and output nodes are the same node.
Such networks seem to occur frequently in metal ion homeostasis.  
For example, see Chifman~\etal~\cite{chifman2012,chifman2017} for iron homeostasis, Cui and Kaandorp~\cite{cui2006} for calcium homeostasis, and Claus~\etal~\cite{claus2015} for zinc homeostasis.

Finally, there are two generalizations of homeostasis theory that are motivated by codimension arguments in bifurcation theory.  The first is chair homeostasis where the homeostasis in input-output functions is flatter than expected.  See Nijhout~\etal~\cite{NBR14}, Golubitsky and Stewart~\cite{GS18}, and Reed~\etal~\cite{RBGSN17}. The second is mode interaction where two infinitesimal homeostasis types occur at the same equilibrium. See Duncan~\etal~\cite{DUN23}.  Interestingly, the simultaneous appearance of different homeostasis types leads to bifurcation in the family of equilibria that generates the homeostasis.  An example of this phenomenon is discussed in Duncan and Golubitsky~\cite{DG19}.  A related biochemical example of multiple types of infinitesimal homeostasis occurring on variation of just one parameter is found in Reed~\etal~\cite{RBGSN17}.

An important research direction to pursue is the application of infinitesimal homeostasis to biological problems.
By this we mean the study of an input-output network associated to a mathematical model for a biological phenomenon.
In this context, an application of infinitesimal homeostasis is more than just the computation of the homeostasis types and homeostasis patterns, as we have done with the network from Figure \ref{F:example8}.
In fact, one can produce dozens of examples by going through the $4$-node input-output networks classified in Huang and Golubitsky \cite{HG22}.
In a biological application the purpose is to understand how a biological mechanism gives rise to biological behavior.
This requires deep understanding of the underlying biology.

There are two promising examples that are currently under investigation.
The first is a model of intracellular iron regulation, adapted from Chifman \etal~\cite{chifman2017}, 
that takes into account free iron ions in the cytosol and in the mitochondria \cite{ABGJ2024}.
In this example the input-output network has eight nodes and two homeostasis patterns.
The second is a model of intracellular cholesterol regulation proposed by Pool \etal~\cite{PST2018} 
after model reduction by quasi-steady state approximation \cite{ABGJ2024}.
In this example the input-output network has $12$ nodes and we find $4$ homeostasis types / patterns.

Let us give some details of the analysis of the intracellular cholesterol regulation.
Starting with the reduced non-dimensional system of Pool \etal~\cite[eqs. 36-48]{PST2018} we write down the `generic admissible system', that is, the most general system of ODEs that have the same state variables on the right handed side as the original system.
From the generic admissible system it is easy to obtain the input-output network $\GG$ -- including the input parameter, the input node and the output node, see Figure \ref{FIG:cholesterol-network}.
The names of the nodes are the state variables defined in Pool \etal~\cite{PST2018} and they reflect the biology of the intracellular cholesterol signaling network.
The next step is to construct the homeostasis pattern network $\PP$ from the input-output network $\GG$, see Figure \ref{FIG:Homeostasis_Pattern_Network}.
Note that $\GG$ has four homeostasis subnetworks.
Finally, using this paper, we can compute the homeostasis patterns on $\PP$, see Figure \ref{FIG:homeostasis_patterns_cholesterol}.
In Table \ref{TAB:homeostasis_patterns_cholesterol} we list the nodes in $\GG$ that are simultaneously homeostatic in each homeostasis pattern. 
In the terminology of Golubitsky and Stewart~\cite{GS23} this is the `model independent' part of the analysis.

\begin{figure}[!htb]
\centering
\includegraphics[scale=0.5]{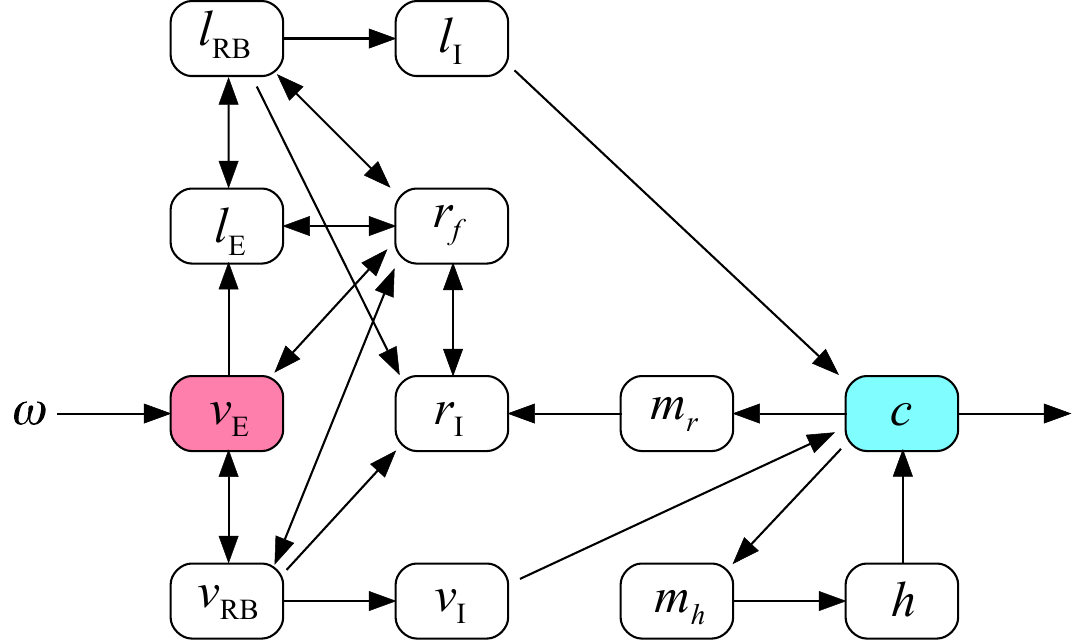}
\caption{Input-output network $\GG$ associated to the intracellular cholesterol regulation model from Pool \etal~\cite{PST2018}.
The input parameter $\omega$ affects the input node $v_E$. The output node is $c$.
\label{FIG:cholesterol-network}}
\end{figure}

\begin{figure}[!htb]
\centering
\begin{tikzpicture}
\node[]                               (0) {Input};
\node[state,right=of 0]               (1) {$\VE$};
\node[state,right=of 1]               (2) {$\widetilde{\mathcal{L}}_1$};
\node[state,right=of 2]               (3) {$\C$};
\node[right=of 3]                     (4) {Output};
\node[state,below=of 3]               (5) {$\widetilde{\mathcal{A}}_1$};
\node[state,right=of 5]               (6) {$\widetilde{\mathcal{A}}_2$};
\node[state,left=of 5]                (7) {$\widetilde{\mathcal{A}}_3$};

\draw[
    >=latex,
    auto=right,                      
    loop above/.style={out=75,in=105,loop},
    every loop,
    ]
    (0) edge node {} (1)
    (1) edge node {} (2)
    (2) edge node {} (3)
    (3) edge node {} (4)
    (3) edge node {} (5)
    (5) edge node {} (6)
    (6) edge node {} (3)
    (5) edge node {} (3)
    (3) edge node {} (6)
    (3) edge node {} (7)
    (7) edge node {} (2);
\end{tikzpicture}
\caption{Homeostasis pattern network $\PP$ associated to the $12$-node input-output network $\GG$ from Figure \ref{FIG:cholesterol-network}. The only structural pattern subnetwork is $\widetilde{\mathcal{L}}_1 = \{ \MR, \IE, \IRB, \Ii, \VRB, \VI, \RF, \RI \}$. 
The appendage pattern subnetworks are
$\widetilde{\mathcal{A}}_1 = \{ \MH \}$, $\widetilde{\mathcal{A}}_2 = \{ \h \}$, $\widetilde{\mathcal{A}}_3 = \{ \MR \}$.
}
\label{FIG:Homeostasis_Pattern_Network}
\end{figure}

\begin{table}[!htb]
\begin{center}
\begin{tabular}{|c|c|c|c|}
\hline
Homeostasis Type & Homeostasis Pattern &  Figure~\ref{FIG:homeostasis_patterns_cholesterol} \\
\hline
$\mathcal{L}_1$ & $\{ \C, \MH, \h, \MR \}$  &(a)  \\
$\mathcal{A}_1$ & $\{ \C, \MR \}$ & (b)\\
$\mathcal{A}_2$ & $\{ \C, \MH, \MR \}$ & (c) \\
$\mathcal{A}_3$ & $\{ \C, \MH, \h \}$ & (d) \\
\hline
\end{tabular}
\end{center}
\caption{Infinitesimal homeostasis patterns and their corresponding homeostsis subnetworks.}
\label{TAB:homeostasis_patterns_cholesterol}
\label{table1} 
\end{table}

\begin{figure}[!htb]
\begin{subfigure}[c]{0.5\textwidth}
\centering
\begin{tikzpicture}
\node[state,right=of 0]               (1) {$\VE$};
\node[state,fill=red,right=of 1]               (2) {$\mathcal{L}_1$};
\node[state,fill=cyan,right=of 2]               (3) {$\C$};
\node[state,fill=cyan,below=of 3]               (5) {$\mathcal{A}_1$};
\node[state,fill=cyan,right=of 5]                (6) {$\mathcal{A}_2$};
\node[state,fill=cyan,left=of 5]                (7) {$\mathcal{A}_3$};

\draw[
    >=latex,
    auto=right,                      
    loop above/.style={out=75,in=105,loop},
    every loop,
    ]
    (1) edge node {} (2)
    (2) edge node {} (3)
    (3) edge node {} (5)
    (5) edge node {} (6)
    (6) edge node {} (3)
    (5) edge node {} (3)
    (3) edge node {} (6)
    (3) edge node {} (7)
    (7) edge node {} (2);
\end{tikzpicture}
\caption{Pattern from $\mathcal{L}_1$}
\medskip\medskip
\end{subfigure}
\begin{subfigure}[c]{0.5\textwidth}
\centering
\begin{tikzpicture}
\node[state,right=of 0]               (1) {$\VE$};
\node[state,right=of 1]               (2) {$\mathcal{L}_1$};
\node[state,fill=cyan,right=of 2]               (3) {$\C$};
\node[state,fill=red,below=of 3]               (5) {$\mathcal{A}_1$};
\node[state,right=of 5]                (6) {$\mathcal{A}_2$};
\node[state,fill=cyan,left=of 5]                (7) {$\mathcal{A}_3$};

\draw[
    >=latex,
    auto=right,                      
    loop above/.style={out=75,in=105,loop},
    every loop,
    ]
    (1) edge node {} (2)
    (2) edge node {} (3)
    (3) edge node {} (5)
    (5) edge node {} (6)
    (6) edge node {} (3)
    (5) edge node {} (3)
    (3) edge node {} (6)
    (3) edge node {} (7)
    (7) edge node {} (2);
\end{tikzpicture}
\caption{Pattern from $\mathcal{A}_1$}
\medskip\medskip
\end{subfigure}
\begin{subfigure}[c]{0.5\textwidth}
\centering
\begin{tikzpicture}
\node[state,right=of 0]               (1) {$\VE$};
\node[state,right=of 1]               (2) {$\mathcal{L}_1$};
\node[state,fill=cyan,right=of 2]               (3) {$\C$};
\node[state,fill=cyan,below=of 3]               (5) {$\mathcal{A}_1$};
\node[state,fill=red,right=of 5]                (6) {$\mathcal{A}_2$};
\node[state,fill=cyan,left=of 5]                (7) {$\mathcal{A}_3$};

\draw[
    >=latex,
    auto=right,                      
    loop above/.style={out=75,in=105,loop},
    every loop,
    ]
    (1) edge node {} (2)
    (2) edge node {} (3)
    (3) edge node {} (5)
    (5) edge node {} (6)
    (6) edge node {} (3)
    (5) edge node {} (3)
    (3) edge node {} (6)
    (3) edge node {} (7)
    (7) edge node {} (2);
\end{tikzpicture}
\caption{Pattern from $\mathcal{A}_2$}
\medskip\medskip
\end{subfigure}
\begin{subfigure}[c]{0.5\textwidth}
\centering
\begin{tikzpicture}
\node[state,right=of 0]               (1) {$\VE$};
\node[state,right=of 1]               (2) {$\mathcal{L}_1$};
\node[state,fill=cyan,right=of 2]               (3) {$\C$};
\node[state,fill=cyan,below=of 3]               (5) {$\mathcal{A}_1$};
\node[state,fill=cyan,right=of 5]                (6) {$\mathcal{A}_2$};
\node[state,fill=red,left=of 5]                (7) {$\mathcal{A}_3$};

\draw[
    >=latex,
    auto=right,                      
    loop above/.style={out=75,in=105,loop},
    every loop,
    ]
    (1) edge node {} (2)
    (2) edge node {} (3)
    (3) edge node {} (5)
    (5) edge node {} (6)
    (6) edge node {} (3)
    (5) edge node {} (3)
    (3) edge node {} (6)
    (3) edge node {} (7)
(7) edge node {} (2);
\end{tikzpicture}
\caption{Pattern from $\mathcal{A}_2$}
\medskip\medskip
\end{subfigure}
\caption{Homeostasis patterns on network $\PP$. Homeostatic nodes in blue; triggering node in red.}
\label{FIG:homeostasis_patterns_cholesterol}
\end{figure}

Now comes the most interesting and challenging step where one has to resort to the original biological model and the underlying biology to determine which homeostasis patterns occur in the particular model and which are biologically significant.
The development of these ideas would depart from the subject matter of this paper and is deferred to a future publication \cite{ABGJIT2024}.

\paragraph{Acknowledgments.}
We thank Marcus Tindall for helpful discussions.
WD was supported by National Institutes of Health (NIH) award R01 GM126555-01.
FA was supported by Funda\c{c}\~ao de Amparo \`a Pes\-qui\-sa do Estado de S\~ao Paulo (FAPESP) grants 2019/12247-7 and 2019/21181-0.

\end{document}